\documentclass[
  a4paper, 
  reqno, 
  oneside, 
  12pt
]{amsart}

\usepackage[utf8]{inputenc}

\usepackage[dvipsnames]{xcolor} % Must come before tikz!
\colorlet{cite}{red}
\usepackage{tikz}  % for diagrams
\usetikzlibrary{arrows,positioning}
\tikzset{ 
  baseline=-2.3pt,
  text height=1.5ex, text depth=0.25ex,
  >=stealth,
  node distance=2cm,
  mid/.style={fill=white,inner sep=2.5pt},
}

\usepackage{lmodern}
\usepackage{tikz-cd}
\usepackage{amsthm, amssymb, amsfonts}
\usepackage{mathtools}
\usepackage{amsthm, amssymb, amsfonts}	% Standard maths packages.
\usepackage[all]{xy}
\usepackage{graphicx,caption,subcaption}
\usepackage{comment}
\usepackage{braket}  % For nice sets
\usepackage{microtype}
\usepackage[makeroom]{cancel}
\usepackage[%
  bookmarks=true,			% show bookmarks bar?
  unicode=true,			% non-Latin characters in Acrobat’s bookmarks
  pdftitle={Isometric Lie 2-actions},		%
  pdfauthor={Herrera, Valencia},	% 
  pdfkeywords={Isometric actions},	% list of keywords
  colorlinks=true,		% false: boxed links; true: colored links
  linkcolor=Red,			% color of internal links
  citecolor=cite,		% color of links to bibliography
  filecolor=magenta,		% color of file links
  urlcolor=RoyalBlue			% color of external links
]{hyperref}				% One of the most useful packages I have found.
\usepackage{cleveref}

\newtheoremstyle{mydef}
  {}		% Space above environment
  {}		% Space below environment
  {}		% Body font
  {}		% Indent amount (empty = no indent, \parindent = para indent)
  {\scshape}	% theorem head font
  {. }		% Punctuation after heading
  { }		% Space after heading
  {\thmname{#1}\thmnumber{ #2}\thmnote{ #3}}	% Heading spec

\newtheorem{theorem}{Theorem}[section]
\newtheorem*{theorem*}{Theorem}
\newtheorem{proposition}[theorem]{Proposition}
\newtheorem*{proposition*}{Proposition}
\newtheorem{lemma}[theorem]{Lemma}
\newtheorem*{lemma*}{Lemma}
\newtheorem{corollary}[theorem]{Corollary}
\newtheorem*{corollary*}{Corollary}
\theoremstyle{definition}
\newtheorem{definition}[theorem]{Definition}
\newtheorem*{definition*}{Definition}
\newtheorem{example}[theorem]{Example}

\theoremstyle{remark}
\newtheorem{remark}[theorem]{Remark}

\usepackage{tikz}

\author{Juan Sebasti\'an Herrera-Carmona and Fabricio Valencia}
\subjclass[2020]{22A22, 58H05, 58D19}
\address{J. S. Herrera-Carmona, F. Valencia - Instituto de Matem\'atica e Estad\'istica, Universidade de S\~ao Paulo, Rua do Mat\~ao 1010, Cidade Universit\'aria, 05508-090 S\~ao Paulo - Brasil. \newline  
sebast1anherr3ra@gmail.com, fabricio.valencia@ime.usp.br}
\date{\today}
\title{Isometric Lie 2-group actions on Riemannian groupoids}

\begin{document}
\maketitle

\begin{abstract}
We study isometric actions of Lie $2$-groups on Riemannian groupoids by exhibiting some of their immediate properties and implications. Firstly, we prove an existence result which allows both to obtain 2-equivariant versions of the Slice Theorem and the Equivariant Tubular Neighborhood Theorem and to construct bi-invariant groupoid metrics on compact Lie $2$-groups. We provide natural examples, transfer some classical constructions and explain how this notion of isometric $2$-action yields a way to develop a 2-equivariant Morse theory on Lie groupoids. Secondly, we give an infinitesimal description of an isometric Lie $2$-group action. We define an algebra of transversal infinitesimal isometries associated to any Riemannian $n$-metric on a Lie groupoid which in turn gives rise to a notion of geometric Killing vector field on a quotient Riemannian stack. If our Riemannian stack is separated then we prove that the algebra formed by such geometric Killing vector fields is always finite dimensional. 
\end{abstract}

\tableofcontents
\section{Introduction}

Isometric actions of Lie groups on Riemannian manifolds have been widely studied in the literature, as they constitute a powerful tool applied to deal with several interesting problems in both mathematics and physics, yielding important geometric-topological consequences and describing certain physical phenomena in nature. They appear in all branches of science where symmetries preserving length and angle measures play a role. In this paper we propose a generalization of the classical setting of isometric actions by studying a notion of isometric action of a categorified version of a Lie group \cite{BD,BS} on a categorified version of a Riemannian manifold \cite{dHF,dHF2,GGHR,PPT}. More precisely, we are interested in studying isometric actions of (strict) Lie $2$-groups on Riemannian groupoids. Our purpose here is twofold. On the one hand, we will be mainly concerned in studying their existence as well as their immediate implications, exhibiting examples and transferring some classical constructions into such a new framework. As an interesting application we will point out how this notion of isometric $2$-action yields a way to develop a 2-equivariant Morse theory on Lie groupoids. On the other hand, we will bring an infinitesimal description of an isometric Lie $2$-group action. Our approach will lead us to the study of an algebra of transversal infinitesimal isometries associated to any Riemannian $n$-metric on a Lie groupoid, which in turn will give rise to a notion of geometric Killing vector field on a quotient Riemannian stack.

Lie groupoids provide a general framework suitable for working with several classical geometries, as they generalize  manifolds, Lie groups, submersions, Lie group actions, foliations, pseudo-groups, vector and principal bundles, among others, giving lead to a new perspective on classical geometric questions and results. Besides, these objects can be seen as an intermediate step in defining differentiable stacks,
some geometric objects admitting singularities and generalizing both manifolds and orbifolds \cite{dH,MM}. The detailed treatment of the interaction between these two theories provides us with a clean way to perform differential geometry on singular spaces. 

As we mentioned earlier, we want to study some features of a natural notion of isometric action of a Lie $2$-group on a Riemannian groupoid. On the one side, as it was explicitly mentioned by Baez and Lauda in \cite{BD}, the notion of Lie 2-group goes back to Brown and Spencer in \cite{BS} where it became clear that classical group theory is just the beginning of a larger subject that sometimes is called higher-dimensional group theory. In many contexts where we are tempted in using groups to tackle certain symmetries-involved problems, it turns out actually to be more natural to use a richer kind of structure where, in addition to group elements describing symmetries, we also have isomorphisms between these, thus describing symmetries between symmetries. On the other side, Riemannian groupoids recently appear into the picture as a differentiable model allowing to carry out Riemannian geometry techniques over more general singular spaces than orbifolds or leaf spaces of singular foliations. Since the seminal works \cite{dHF,dHF2}, the notion of Riemannian groupoid (stack) defined and studied therein has been applied to satisfactorily extend several theories where Riemannian manifolds have played an important role. For instance, some of the recent contributions in which Riemannian groupoid metrics have been used as a tool to describe topological and geometric features of Lie groupoids as well as their differentiable stacks come in order.
\begin{itemize}
\item Resolutions of proper Riemannian groupoids were introduced in \cite{PTW} in order to obtain a desingularization of their underlying differentiable stacks via a successive blow-up construction.
\item A theory of stacky geodesics on Riemannian stacks was developed in \cite{dHdM} allowing to establish a stacky version
of the Hopf--Rinow Theorem.
\item The problem of understanding invariant linearization of proper Lie groupoids was addressed in \cite{dHdM2} where the authors fixed and extended previous results in the literature as well as provided a sufficient criterion that uses compatible complete metrics and covers the case of proper group actions.
\item A notion of Morse Lie groupoid morphism was introduced in \cite{OV}, allowing to extend the main results of classical Morse theory to the context of Lie groupoids and differentiable stacks.
\end{itemize}

The paper is organized as follows. In Section \ref{S:2} we briefly introduce the necessary terminology and notions about Riemannian groupoids and Lie $2$-groups we will be dealing with along the paper. In Section \ref{S:3} we define our notion of isometric Lie $2$-group action on a Riemannian groupoid:
\begin{definition*}
A $2$-action $\theta=(\theta^1,\theta^0):(G_1\times X_1\rightrightarrows G_0\times X_0)\to(X_1\rightrightarrows X_0)$ of a Lie 2-group $G_1\rightrightarrows G_0$ on a Riemannian groupoid $(X_1\rightrightarrows X_0,\eta)$ is said to be \emph{isometric} if $G_2$ acts by isometries on $(X_2,\eta^{(2)})$.
\end{definition*}

For the sake of simplicity at several stages of our work will be enough to consider only $1$-metrics $\eta^{(1)}$ on $X_1$ so that in those cases we proceed by supposing that $G_1\rightrightarrows G_0$ acts isometrically on $(X_1\rightrightarrows X_0,\eta)$ if $G_1$ acts by isometries on $(X_1,\eta^{(1)})$. It is important to mention that most of our results below are also valid when working with general $n$-metrics on $X_n$ for $n\geq 3$, see Remark \ref{Rmk2}. We show that we can pass to the quotient groupoid metrics by free and proper isometric Lie $2$-group actions and, motivated by the proof of the groupoid linearization theorem provided in \cite{dHF}, we explain how to construct $2$-equivariant weak linearizations around $G_0$-invariant saturated submanifolds in $X_0$. This is the content of Proposition \ref{PQuotient} and Proposition \ref{Lin1}, respectively. Also, by a dual averaging procedure we obtain the following result:

\begin{theorem*}[\ref{Existence2}]
	Let $\theta$ be a $2$-action of a Lie $2$-group $G_1\rightrightarrows G_0$ on a proper groupoid $X_1\rightrightarrows X_0$. If $G_1$ is compact then there always exists a Riemannian groupoid metric on $X_1\rightrightarrows X_0$ for which $\theta$ becomes an isometric $2$-action.
\end{theorem*}

Two interesting applications can be derived from the previous existence result. Firstly, we can prove analogous statements to those of the Slice Theorem and the Equivariant Tubular Neighborhood Theorem in our setting. More precisely, we get that:
\begin{theorem*}[\ref{ETubular}]
	Suppose that $\theta$ is a proper $2$-action of a Lie $2$-group $G_1\rightrightarrows G_0$ on a Riemannian groupoid $(X_1\rightrightarrows X_0,\eta)$. Then, for every $x_0\in X_0$ there exist $\epsilon>0$, another groupoid metric $\tilde{\eta}$ and a $2$-equivariant Lie groupoid isomorphism
	$$\Psi: (G_1\times_{\textnormal{Iso}_{G_1}(1_{x_0})}S_{1_{x_0}}\rightrightarrows G_0\times_{\textnormal{Iso}_{G_0}(x_0)}S_{x_0}) \xrightarrow[]{\cong} (\textnormal{Tub}_{\epsilon}^{\tilde{\eta}}(G_1\cdot 1_{x_0})\rightrightarrows \textnormal{Tub}_{\epsilon}^{\tilde{\eta}}(G_0\cdot x_0)).$$
\end{theorem*}

Here $S_{1_{x_0}}\rightrightarrows S_{x_0}$ denotes a Lie subgroupoid slice at $x_0\in X_0$ (see Definition \ref{DefSlice}), which always exists by Proposition \ref{SliceThm} and $\textnormal{Tub}_{\epsilon}^{\tilde{\eta}}(G_1\cdot 1_{x_0})\rightrightarrows \textnormal{Tub}_{\epsilon}^{\tilde{\eta}}(G_0\cdot x_0)$ is the natural Lie groupoid between the tubular neighborhoods of the orbits $G_1\cdot 1_{x_0}$ in $X_1$ and $G_0\cdot x_0$ in $X_0$. Secondly, in Subsection \ref{S:3:1} we define orthogonal Lie $2$-groups (i.e. Lie $2$-group endowed with bi-invariant Riemannian groupoid metrics), thus proving that:
\begin{proposition*}[\ref{Bi-Existence}]
Every Lie $2$-group $G_1\rightrightarrows G_0$ with $G_1$ compact can be equipped with a bi-invariant Riemannian groupoid metric.
\end{proposition*}

We also describe the infinitesimal counterpart of an orthogonal Lie $2$-group, comparing it with an existing notion of $\mathcal{L}$-orthogonal crossed module of Lie algebras already known in the literature \cite{Ba,FMM}. In Subsection \ref{S:3:2} we present several natural examples and transfer two classical constructions called principal connection warpings and Cheeger deformations into our new framework. We also explain how our notion of isometric $2$-action can be used to develop a $2$-equivariant analogous of the results proved in \cite{OV} which are aimed at extending classical Morse theory to the context of Lie groupoids and their differentiable stacks, see Example \ref{2-EquiMorse}. More precisely, let $f:X_0\to \mathbb{R}$ be a $G_0$-invariant basic function which satisfies a Morse--Bott (normal) nondegenerate condition along its critical submanifolds. Thus, based on the results proved in \cite{OV} we indicate how to get a $2$-equivariant version of the Morse lemma around a $G_0$-invariant nondegenerate critical saturated submanifold of $f$, how to describe the topological behavior of the level subgroupoids of $f$ in a $2$-equivariant way, and also how to construct an equivariant Morse--Bott double cochain complex which computes the equivariant cohomology of a 2-action as defined in \cite{OBT}.

In Section \ref{S:4} we give an infinitesimal description of an isometric Lie $2$-group action. In order to do so we define the Lie 2-group of strong isometries of $(X_1\rightrightarrows X_0,\eta)$ which is associated to a natural crossed module structure $(\mathrm{Iso}(X,\eta),\textnormal{Bis}_{\eta}(X),I,\alpha)$ introduced in Proposition \ref{IsoStrong1}. The infinitesimal counterpart of such a Lie $2$-group of isometries is the Lie $2$-algebra of strong multiplicative Killing vector fields of $(X_1\rightrightarrows X_0,\eta)$ which is the defined by means of the crossed module structure $(\mathfrak{o}_{m}(X),\Gamma_{\eta}(A_X),\delta,D)$ from Proposition \ref{Strongkilling}. By using those crossed module structures we show that:

\begin{theorem*}[\ref{InfinitesimalIsometriesRep}]
	Let $\theta$ be an isometric right 2-action of a Lie $2$-group $G_1\rightrightarrows G_0$ on a Riemannian groupoid $(X_1\rightrightarrows X_0,\eta)$. Denote by $(G,H,\rho,\alpha)$ to the crossed module of Lie groups associated to $G_1\rightrightarrows G_0$ and by $(\mathfrak{g},\mathfrak{h},\partial,\mathcal{L})$ to its corresponding crossed module of Lie algebras. Then:
	\begin{itemize}
		\item there is a morphism of crossed modules of Lie groups $$(\sigma,\Sigma):(G,H,\rho,\alpha) \to (\textnormal{Iso}(X,\eta),\textnormal{Bis}_\eta(X),I,\alpha),$$
		that is defined as $\sigma_h(x)=1_xh$ and $\Sigma_g(p)=p1_g$, and
		\item there is a morphism of crossed modules of Lie algebras $$(j_{-1},j_0):(\mathfrak{g},\mathfrak{h},\partial,\mathcal{L})\to (\mathfrak{o}_{m}(X),\Gamma_{\eta}(A_X),\delta,D),$$
		which is defined as $j_{-1}=\left.\tilde{\xi}\right|_{X_0}$ and  $j_0(\zeta)=(\tilde{1_{\zeta}},\tilde{\zeta})$ for all $\xi \in \mathfrak{h}$ and $\zeta \in \mathfrak{g}$.
	\end{itemize}
\end{theorem*}

Here the symbol $\tilde{}$ stands to denote the fundamental vector field associated to an element in a Lie algebra with respect to a given Lie group action. It is important to
mention that our infinitesimal description is motivated by those results developed in \cite{OW} in order to describe the Lie 2-algebra structure that the set of multiplicative vector fields on a Lie groupoid has. In Subsection \ref{S:4:1} we weaken the condition for a diffeomorphism to be an isometry by imposing instead a transversal isometric condition along groupoid orbits. This leads us to define the Lie $2$-algebra of weak Killing multiplicative vector fields associated to the crossed module $(\mathfrak{o}_{m}^{\textnormal{w}}(X),\Gamma(A_X),\delta,D)$ which is a well defined object for any Riemannian groupoid $n$-metric. By adapting some of the results proved in \cite{OW} to our context, in Subsection \ref{S:4:2} we get that:
\begin{theorem}[\ref{MoritaTheorem}]
	If $(X_1 \rightrightarrows X_0,\eta^X)$ and $(Y_1 \rightrightarrows Y_0,\eta^Y)$ are Morita equivalent Riemannian groupoids then the crossed modules $(\mathfrak{o}_{m}^{\textnormal{w}}(X),\Gamma(A_X),\delta,D)$ and $(\mathfrak{o}_{m}^{\textnormal{w}}(Y),\Gamma(A_Y),\delta,D)$ are isomorphic in the derived category of crossed modules. In consequence, the following quotient spaces are isomorphic as Lie algebras:
	$$\mathfrak{o}_{m}^{\textnormal{w}}(X)/\textnormal{im}(\delta)\cong \mathfrak{o}_{m}^{\textnormal{w}}(Y)/\textnormal{im}(\delta).$$
\end{theorem}

The elements of the Lie algebra $\mathfrak{o}_{m}^{\textnormal{w}}(X)/\textnormal{im}(\delta)$ are called geometric Killing vector fields. It is worth mentioning that such a notion of Killing vector field recovers the classical notions of Killing vector field on both a Riemannian manifold and a Riemannian orbifold as defined for instance in \cite{BZ} as well as the notion of transverse Killing vector field on a regular
Riemannian foliation as defined in \cite[p. 84]{Mo}. If $([X_0/X_1],[\eta])$ denotes the quotient Riemannian stack presented by $(X_1 \rightrightarrows X_0,\eta)$ then by using the desingularization theorem for proper Riemannian groupoids proved in \cite{PTW} we show:

\begin{theorem}[\ref{FiniteThm2}]
	Let $(X_1 \rightrightarrows X_0,\eta)$ be a proper Riemannian groupoid with compact orbit space $X_0/X_1$. Then the algebra of geometric Killing vector fields on $([X_0/X_1],[\eta])$ has finite dimension.
\end{theorem}

Finally, by averaging with respect to a proper Haar measure system for $(X_1 \rightrightarrows X_0,\eta)$ we prove that there exists a projection from what we call the space of projectable vector fields of Killing type to the space of weak multiplicative Killing vector fields, see Theorem \ref{ExistenceWeak}. Hence, motivated by some results proved in \cite{CMS} we give another face to the space of geometric Killing vectors fields by using Killing invariant sections, thus obtaining that Theorem \ref{ExistenceWeak} provides us with a method to construct geometric Killing vector fields over the quotient Riemannian stack $([X_0/X_1],[\eta])$.

\medskip

{\bf Acknowledgments:} We would like to thank Mateus de Melo, Cristian Ortiz and Luca Vitagliano for several enlightened discussions, comments, and suggestions which allowed us to improve this work. Valencia is grateful to Matias del Hoyo for inviting and supporting him to visit IMPA in Rio de Janeiro in December, 2022. Their mathematical discussions during the academic visit encouraged the authors to start enhancing the presentation and the results stated in the first version of the present work. We are also thankful to the anonymous referees who provided many suggestions and corrections that improved the quality of our results. Herrera-Carmona was supported by the Brazilian Federal Foundation CAPES - Finance Code 001. Valencia was supported by Grant 2020/07704-7 Sao Paulo Research Foundation - FAPESP.

\section{Preliminaries}\label{S:2}
In this short section we briefly introduce the basics on Lie groupoids we shall be using throughout. For specific details the reader is recommended to visit for instance \cite{BD,BS,dH,dHF,Ma}. 

A \emph{Lie groupoid} $X_1\rightrightarrows X_0$ consists of a manifold $X_0$ of objects and a manifold $X_1$ of arrows, two surjective submersions $s_X,t_X:X_1\to X_0$ respectively indicating the source and the target of the arrows, and a smooth associative composition $m_X:X_2\to X_1$ over the set of composable arrows $X_2=X_1\times_{X_0} X_1$, admitting unit $u_X:X_0\to X_1$ and inverse $i_X:X_1\to X_1$, subject
to the usual groupoid axioms. The collection of maps mentioned above are called \emph{structural maps} of the Lie groupoid $X_1\rightrightarrows X_0$. We shall drop up the sub-index notation for the structural maps only if there is no risk of confusion. Special instances of Lie groupoids are given by manifolds, Lie groups, Lie group actions, surjective submersions, foliations, pseudogroups, principal bundles, vector bundles, among others \cite{dH, Ma}.

Let $X_1\rightrightarrows X_0$ be a Lie groupoid. For each $x\in X_0$, its \emph{isotropy group} $X_x:=s_X^{-1}(x)\cap t_X^{-1}(x)$ is a Lie group and an embedded submanifold in $X_1$. There is an equivalence relation in $X_0$ defined by $x\sim y$ if there exists $p\in X_1$ with $s_X(p)=x$ and $t_X(p)=y$. The corresponding equivalence class of $x\in X_0$ is denoted by $\mathcal{O}_x\subseteq X_0$ and called the \emph{orbit} of $x$. The previous equivalence relation defines a quotient space $X_0/X_1$ called the \emph{orbit space} of $X_1\rightrightarrows X_0$.  This space equipped with the quotient topology is in general a \emph{singular space}, that is, it does not carry a differentiable structure making the quotient projection $X_0\to X_0/X_1$ a surjective submersion. 

A \emph{Lie groupoid morphism} between two Lie groupoids $X_1\rightrightarrows X_0$ and $X_1'\rightrightarrows X_0'$ is a pair $\phi:=(\phi^1,\phi^0)$ where $\phi^1:X_1\to X_1'$ and $\phi^0:X_0\to X_0'$ are smooth maps commuting with both source and target maps and preserving the composition maps. The \emph{tangent groupoid} $TX_1\rightrightarrows TX_0$ is obtained by applying the tangent functor at both manifolds of arrows and objects and all of its structural maps. If $S \subset X_0$ is a saturated manifold (only an orbit for instance) then we can restrict the groupoid structure to $X_{S}=s_X^{-1}(S)=t_X^{-1}(S)$, thus obtaining a Lie subgroupoid $X_S\rightrightarrows S$ of $X_1\rightrightarrows X_0$. Furthermore, the Lie groupoid structure of $TX_1\rightrightarrows TX_0$ can be restricted to define a new Lie groupoid $\nu(X_S)\rightrightarrows \nu(S)$ having the property that all of its structural maps are fiberwise isomorphisms. The \emph{Lie algebroid} associated to the Lie groupoid $X_1\rightrightarrows X_0$ is defined to be the vector bundle $A_X:=\textnormal{ker}(ds_X)|_{X_0}\subset TX_1$ with anchor map $\rho:A_X\to TX_0$ obtained by restricting $dt_X:TX_1\to TX_0$ to $A_X$.

\subsection{Riemannian groupoids}
The notion of Riemannian metric on a Lie groupoid that we will be dealing with along the paper was introduced in \cite{dHF} (see also \cite{GGHR,PPT}). Such a notion is compatible with the groupoid composition so that it plays an important role in several parts of our work. We start by recalling that a submersion $\pi:(E,\eta^E)\to B$ with $(E,\eta^E)$ a Riemannian manifold is said to be \emph{Riemannian} if the fibers of it are equidistant (transverse condition). In this case the base $B$ gets an induced metric $\eta^B:=\pi_\ast \eta^E$ for which the linear map $d\pi(e):(\textnormal{ker}(d\pi(e)))^\perp\to T_{\pi(e)}B$ is an isometry for all $e\in E$. If $(\eta^{E})^\ast$ denotes the dual metric associated to $\eta^{E}$ then the condition for a Riemannian submersion can be rephrased as follows. For all $e\in E$ the map  $d\pi(e)^\ast: T_{\pi(e)}^\ast B \to \textnormal{ker}(d\pi(e))^\circ$ is an isometry, where $\textnormal{ker}(d\pi(e))^\circ$ denotes the annihilator of the vectors tangent to the fiber. If $\pi:E\to B$ is a surjective submersion then a Riemannian metric $\eta^E$ on $E$ is said to be \emph{transverse} to $\pi$ if for all $x\in B$ and all $e_1,e_2\in \pi^{-1}(x)$ we have that the map
$$d\pi(e_1)^\ast\circ (d\pi(e_2)^\ast)^{-1}:\textnormal{ker}(d\pi(e_2))^\circ\to T_x^\ast B\to  \textnormal{ker}(d\pi(e_1))^\circ,$$
is a linear isometry. In this case, there exists a unique metric $\eta^B$ on $B$ such that $\pi$ becomes a Riemannian submersion. Such a metric is defined by the expression $\eta^B(d\pi(v), d\pi(w)) := \eta^E(v, w)$ for $v, w \in \textnormal{ker}(d\pi)^\perp$ and is called the \emph{push-forward metric}. The notation we shall be using for the previous Riemannian metric is $\eta^B:=\pi_\ast \eta^E$.

It is well known that given a Lie groupoid $X_1\rightrightarrows X_0$ every pair of composable arrows in $X_2$ may be identified with an element in the space of commutative triangles so that it admits an action of $S_3$ determined by permuting the vertices of such triangles. In these terms, a \emph{Riemannian groupoid} is a pair $(X_1\rightrightarrows X_0,\eta)$ where $X_1\rightrightarrows X_0$ is a Lie groupoid and $\eta=\eta^{(2)}$ is a Riemannian metric on $X_2$ that is invariant by the $S_3$-action and transverse to the composition map $m_X:X_2\to X_1$. The metric $\eta^{(2)}$ induces metrics $\eta^{(1)}=((\pi_2)_X)_\ast \eta^{(2)}=(m_X)_\ast \eta^{(2)}=((\pi_1)_X)_\ast \eta^{(2)}$ on $X_1$ and $\eta^{(0)}=(s_X)_\ast \eta^{(1)}=(t_X)_\ast \eta^{(1)}$ on $X_0$ such that $(\pi_2)_X, m_X, (\pi_1)_X:X_2\to X_1$ and $s_X,t_X:X_1\to X_0$ are Riemannian submersions and $i_X:X_1\to X_1$ is an isometry. This is because the $S_3$-action permutes these face maps.

The metric $\eta^{(j)}$, for $j=2,1,0$, is called a $j$-\emph{metric}. It is important to mention that every proper groupoid can be endowed with a $2$-metric (more generally, an $n$-metric as defined below) and if a Lie groupoid admits a $2$-metric then it is weakly linearizable around any saturated submanifold. For more details visit \cite{dHF}.
\begin{remark}
	The notion of \emph{$n$-metric} on Lie groupoids for $n\geq 3$ was introduced in \cite{dHF}. This is just a Riemannian metric on the set of $n$-composable arrows $X_n$ that is invariant by the canonical $S_{n+1}$-action on $X_n$ and transverse to one (hence to all) face map $X_n\to X_{n-1}$. We can push this $n$-metric forward with the different face maps $X_n\to X_{n-1}$ to define an $(n-1)$-metric on $X_{n-1}$ in such a way these face maps become Riemannian submersions. One can use this process to obtain $r$-metrics $\eta^{(r)}$ on $X_{r}$ for all $0\leq r\leq n-1$ so that we get Riemannian submersions $(X_{r},\eta^{(r)})\to (X_{r-1},\eta^{(r-1)})$.
\end{remark}

\subsection{Lie 2-groups}
A \emph{Lie 2-group} is a group internal to the category of Lie groupoids. In other words, a Lie 2-group is a Lie groupoid $G_1\rightrightarrows G_0$ where both $G_1$ and $G_0$ are Lie groups and the structural maps of $G_1\rightrightarrows G_0$ are Lie group homomorphisms \cite{BD,BS}.  We will denote by $\ast$ the composition of arrows in $G_1\rightrightarrows G_0$ and by $\cdot$ the product of arrows in $G_1$. For instance, the fact that the multiplication is a morphism of groups amounts to the identity:
\begin{equation}\label{MultiProduct}
	(g_1\ast g_2)\cdot(g_1'\ast g_2')=(g_1\cdot g_1')\ast(g_2\cdot g_2'),\quad \forall (g_1,g_2),(g_1',g_2')\in G_2,
\end{equation}
where $G_2$ is the space of pairs of composable arrows of $G_1\rightrightarrows G_0$.

It is well known that there are several alternative ways to think of Lie 2-groups. One of them is described in terms of crossed modules. Recall that a \emph{crossed module of Lie groups} is a quadruple $(G,H,\rho,\alpha)$ consisting of a Lie group homomorphism $\rho:H\to G$ together with an action $\alpha$ of $G$ on $H$, $(g,h)\mapsto g h$, by Lie group automorphisms such that:
\[ \rho(g h)=g \rho(h) g^{-1},\quad \rho(h)h'=h h' h^{-1}, \quad g\in G,\ h,h'\in H. \]

To such a crossed module one associates a Lie $2$-group with objects the Lie group $G_0:=G$, arrows the semi-direct product $G_1:=H\rtimes G$ so that:
\[ (h_1,g_1)\cdot (h_2,g_2):=(h_1 (g_1h_2), g_1 g_2), \]
and structure maps:
\[ s_G(h,g)=g,\quad t_G(h,g)=\rho(h)g, \quad (h_1,\rho(h_2)g_2)\ast (h_2,g_2)=(h_1h_2,g_2). \]
Conversely, any Lie 2-group $G_1 \rightrightarrows G_0$ has an associated crossed module of Lie groups $(G,H,\rho,\alpha)$, where $G=G_0$, $H=\textnormal{ker}(s_G)$, $\rho:=t_G|_H:H\to G$ and $G$ acts on $H$ by conjugation via the identity bisection: $g h:=1_g\cdot h\cdot 1_{g^{-1}}$. This defines an equivalence of categories between the category of Lie 2-groups and the category of crossed modules of Lie groups \cite{BD}.

A \emph{Lie 2-algebra} is a Lie algebra internal to the category of Lie groupoids. More concretely, a Lie 2-algebra is a Lie groupoid $\mathfrak{g}_1\rightrightarrows \mathfrak{g}_0$ where both $\mathfrak{g}_1$ and $\mathfrak{g}_0$ are Lie algebras and the structural maps of $\mathfrak{g}_1\rightrightarrows \mathfrak{g}_0$ are Lie algebra homomorphisms. As expected, the Lie functor provides a correspondence between Lie 2-groups and Lie 2-algebras.

The infinitesimal object associated to a crossed module of Lie groups is the so-called \emph{crossed module of Lie algebras}. This is a quadruple $(\mathfrak{g},\mathfrak{h},\partial,\mathcal{L})$ where $\mathfrak{g}$ and $\mathfrak{h}$ are Lie algebras and $\partial:\mathfrak{h}\to\mathfrak{g}$ and $\mathcal{L}:\mathfrak{g}\to\textnormal{Der}(\mathfrak{h})$ are Lie algebra homomorphisms verifying
\begin{equation}\label{PeifferLieAlgebra}
\partial(\mathcal{L}_xy)=[x,\partial(y)]_\mathfrak{g}\quad\textnormal{and}\quad \mathcal{L}_{\partial(x)}y=[x,y]_\mathfrak{h}.
\end{equation}

A Lie $2$-algebra $\mathfrak{g}_1\rightrightarrows \mathfrak{g}_0$ has an associated crossed module given by the data $\mathfrak{h}=\textnormal{ker}(s)$, $\mathfrak{g}=\mathfrak{g}_0$, $\partial=t|_\mathfrak{h}$ and $\mathcal{L}_x=\textnormal{ad}^{1}_{u(x)}$ for all $x\in \mathfrak{g}$. Conversely, a crossed module of Lie algebras $(\mathfrak{g},\mathfrak{h},\partial,\mathcal{L})$ has an associated Lie $2$-algebra which in turn is given by the data $\mathfrak{g}_1=\mathfrak{h}\rtimes \mathfrak{g}$ with Lie algebra structure provided by the semi-direct product with respect to $\mathcal{L}$, $\mathfrak{g}_1=\mathfrak{g}$, and structural maps
$$s(x,y)=y,\quad t(x,y)=\partial(x)+y,\quad u(y)=(0,y),\quad i(x,y)=(-x,y+\partial(x)), $$
$$m((x',y+\partial(x)),(x,y))=(x+x',y).$$

A \emph{morphism} of crossed modules  $f:(\mathfrak{g},\mathfrak{h},\partial,\mathcal{L})\to (\mathfrak{g}',\mathfrak{h}',\partial',\mathcal{L}')$ consists of two Lie algebra homomorphisms $f_0:\mathfrak{g}\to \mathfrak{g}'$ and $f_1:\mathfrak{h}\to \mathfrak{h}'$ such that $f_0\circ \partial =\partial'\circ f_1$ and $f_1(\mathcal{L}_xy)=\mathcal{L}'_{f_0(x)}(f_1(y))$. This induces a pair of Lie algebra homomorphisms $\textnormal{ker}(\partial)\to \textnormal{ker}(\partial')$ and $\textnormal{coker}(\partial)\to \textnormal{coker}(\partial')$. A morphism of crossed modules is a \emph{quasi-isomorphism} if both of these Lie algebra morphisms are isomorphisms. The \emph{derived category} of crossed modules of Lie algebras is defined to be the localization of the category of crossed modules of Lie algebras obtained by inverting all quasi-isomorphisms.

\section{Isometric Lie 2-group actions}\label{S:3}
In the sequel we shall denote by $G_1\rightrightarrows G_0$ a Lie 2-group and by $(X_1\rightrightarrows X_0,\eta)$ a Riemannian groupoid. A \emph{left 2-action} of $G_1\rightrightarrows G_0$ on $X_1\rightrightarrows X_0$ is defined to be a Lie groupoid morphism $\theta=(\theta^1,\theta^0):(G_1\times X_1\rightrightarrows G_0\times X_0)\to(X_1\rightrightarrows X_0)$ such that both maps $\theta^1$ and $\theta^0$ are usual left Lie group actions. Here $G_1\times X_1\rightrightarrows G_0\times X_0$ denotes the product Lie groupoid. We shall say that $\theta$ is a free (resp. proper) $2$-action if both $\theta^1$ and $\theta^0$ are free (resp. proper) actions. Note that given a left 2-action of $G_1\rightrightarrows G_0$ on $X_1\rightrightarrows X_0$ we immediately get that the structural maps of $X_1\rightrightarrows X_0$ are equivariant with respect the structural maps of $G_1\rightrightarrows G_0$. More precisely, if $p\in X_1, x\in X_0$ and $g\in G_1, g_0\in G_0$ then we obtain that
\begin{equation}\label{Equivariant}
s_X(gp)=s_G(g)s_X(p),\quad t_X(gp)=t_G(g)t_X(p),\quad u_X(g_0x)=u_G(g_0)u_X(x).
\end{equation}

Moreover, we have that the action on arrows is \emph{multiplicative}, meaning that for pairs of composable arrows $(p,q)\in X_2$ and $(g,h)\in G_2$ the following formula holds true
\begin{equation}\label{MultAction}
	(p*q)(g*h)=(pg)*(qh).
\end{equation} 

Here we are denoting by $m_X(p,q)=p*q$ and $m_G(g,h)=g*h$. It is clear that $G_2$ is also a Lie group with the structure induced from the direct product $G_1\times G_1$ and that the $2$-action $\theta$ indices a canonical left action $\theta^2$ of $G_2$ on $X_2$.

\begin{definition}\label{MainDefinition}
A $2$-action of $G_1\rightrightarrows G_0$ on $(X_1\rightrightarrows X_0,\eta)$ is said to be \emph{isometric} if $G_2$ acts by isometries on $(X_2,\eta^{(2)})$.
\end{definition}
An immediate consequence of the previous definition is the following.
\begin{lemma}\label{Rmk1}
The action of $G_1$ on $(X_1,\eta^{(1)})$ is by isometries. Consequently, the action of $G_0$ on $(X_0,\eta^{(0)})$ is also by isometries.
\end{lemma}
\begin{proof}
For simplicity we shall verify that if $G_1$ acts on $(X_1,\eta^{(1)})$ by isometries then $G_0$ acts on $(X_0,\eta^{(0)})$ by isometries as well. Let $v,w\in T_{x}X_0$ and $g_0\in G_0$. It is clear that there are $\tilde{v},\tilde{w}\in \textnormal{ker}(d(s_X)_p)^{\perp_{\eta^{(1)}}}$ with $s_X(p)=x$ such that $d(s_X)_p(\tilde{v})=v$ and $d(s_X)_p(\tilde{w})=w$ and there is $g\in G_1$ such that $s_G(g)=g_0$. Therefore,
\begin{eqnarray*}
(\theta^0_{g_0})_x^\ast\eta^{(0)}(v,w) & = & \eta^{(0)}_{g_0x}(d(\theta^0_{g_0})_x(v),d(\theta^0_{g_0})_x(w))=\eta^{(0)}_{g_0x}(d(\theta^0_{g_0}\circ s_X)_p(\tilde{v}),d(\theta^0_{g_0}\circ s_X)_p(\tilde{w}))\\
\star& = & \eta^{(0)}_{g_0x}(d(s_X\circ \theta^1_{g})_p(\tilde{v}),d(s_X\circ \theta^1_{g})_p(\tilde{w}))=\eta^{(1)}_{gp}(d(\theta^1_{g})_p(\tilde{v}),d(\theta^1_{g})_p(\tilde{w}))\\
& = & \eta^{(1)}_{p}(\tilde{v},\tilde{w})=\eta^{(0)}_{x}(v,w).
\end{eqnarray*}
In the equality $\star$ above we used the fact that $s_X$ is a Riemannian submersion and that the action $\theta^1$ preserves the horizontal distribution $\textnormal{ker}(ds_X)^{\perp_{\eta^{(1)}}}$. Note that this computation does not depend on the choice of $\tilde{v},\tilde{w},p$ and $g$. Furthermore, we may obtain the same conclusion by choosing $\eta^{(0)}=(t_X)_\ast \eta^{(1)}$ instead of $\eta^{(0)}=(s_X)_\ast \eta^{(1)}$. The fact that $G_1$ acts on $(X_1,\eta^{(1)})$ by isometries provided that $G_2$ acts on $(X_2,\eta^{(2)})$ by isometries can be similarly shown by using that $\eta^{(1)}=(\pi_{2,X})_\ast \eta^{(2)}=(m_X)_\ast \eta^{(2)}=(\pi_{1,X})_\ast \eta^{(2)}$ and that $\pi_{2,X}, m_X, \pi_{1,X}:X_2\to X_1$ are Riemannian submersions verifying analogous equivariant relations as those in Equation \eqref{Equivariant} with respect to the maps $\pi_{2,G}, m_G, \pi_{1,G}:G_2\to G_1$.
\end{proof}

\begin{remark}\label{Rmk2}
More generally, by using exactly the same arguments as in Lemma \ref{Rmk1} it is simple to check that if $\eta^{(n)}$ is an $n$-metric on $X_n$ and the induced left action $\theta^n$ of $G_n$ on $(X_n,\eta^{(n)})$ is by isometries then the action $\theta^k$ of $G_k$ on $(X_k,\eta^{(k)})$ will be by isometries for all $0\leq k\leq n-1$.

%Conversely, if we suppose that $G_0$ and $G_1$ act by isometries on $(X_0,\eta^{(0)})$ and $(X_1,\eta^{(1)})$, respectively, then  the induced action of $G_n$ on $(X_n,\eta^{(n)})$ is by isometries for all $n\geq 2$. Last assertion follows by using the formula
%\begin{equation}\label{delHoyoFernandesFormula}
%\eta^{(2)}=(\pi_1)_X^\ast \eta^{(1)}+(\pi_2)_X^\ast \eta^{(1)}-s_X^\ast \eta^{(0)}\circ (d(\pi_1)_X\oplus d(\pi_1)_X),
%\end{equation}
%which is derived from Remark 2.5 from \cite{dHF} once the $n$-metric is chosen. Indeed, take $(g,h)\in G_2$ and $(v,w),(v',w')\in T_{(p,q)}X_2$ such that $d(s_X)_p(v)=d(t_X)_q(w)$ and $d(s_X)_p(v')=d(t_X)_q(w')$. It is clear that $d(\theta^2_{(g,h)})_{(p,q)}(v,w)=(d(\theta^1_{g})_{p}(v),d(\theta^1_{g})_{q}(w))$. Therefore, by using Formula \eqref{delHoyoFernandesFormula} we obtain that the expression $\eta^{(2)}_{(gp,hq)}(d(\theta^2_{(g,h)})_{(p,q)}(v,w),d(\theta^2_{(g,h)})_{(p,q)}(v',w'))$ equals
%\begin{eqnarray*}
%& &\eta^{(1)}_{gp}(d(\theta^1_{g})_{p}(v),d(\theta^1_{g})_{p}(v'))+\eta^{(1)}_{hq}(d(\theta^1_{h})_{p}(w),d(\theta^1_{h})_{q}(w'))-\eta^{(0)}_{s_X(gp)}(d(s_X)_{gp}(d(\theta^1_{g})_{p}(v)),d(\theta^1_{g})_{p}(v'))\\
%&=& \eta^{(1)}_{p}(v,v')+\eta^{(1)}_{q}(w,w')-\eta^{(0)}_{s_G(g)s_X(p)}(d(\theta^0_{s_G(g)})_{s_X(p)}(d(s_X)_p(v)),d(\theta^0_{s_G(g)})_{s_X(p)}(d(s_X)_p(v')))\\
%&=& \eta^{(1)}_{p}(v,v')+\eta^{(1)}_{q}(w,w')-\eta^{(0)}_{s_X(p)}(d(d(s_X)_p(v),d(s_X)_p(v'))=\eta^{(2)}_{(p,q)}((v,w),(v',w')).\\
%\end{eqnarray*}
\end{remark}

We warn the reader that for the sake of simplicity at several stages of our study will be enough to consider only $1$-metrics $\eta^{(1)}$ on $X_1$ so that in those cases we proceed by supposing that $G_1\rightrightarrows G_0$ acts isometrically on $(X_1\rightrightarrows X_0,\eta)$ if $G_1$ acts by isometries on $(X_1,\eta^{(1)})$. This is because the computations as well as the arguments will be both canonical and analogous if we consider $2$-metrics or even $n$-metrics in general. Indeed, most of the notions and results we will introduce below shall have an analogous statement if the simplicial approach from Remark \ref{Rmk2} is considered.

Two interesting consequences that come up from our definition of isometric Lie $2$-group action are the following. Firstly, we can pass to the quotient groupoid metrics by free and proper isometric $2$-actions in a natural fashion, thus obtaining examples of Riemannian groupoid submersions as defined in \cite[Def. 3.2.1]{dHF2}.

\begin{proposition}\label{PQuotient}
	If $\theta$ is a free and proper isometric $2$-action of $G_1\rightrightarrows G_0$ on $(X_1\rightrightarrows X_0,\eta)$ then there is a structure of Riemannian groupoid $(X_1/G_1\rightrightarrows X_0/G_0, \overline{\eta})$ so that the canonical projection $\pi=(\pi_1,\pi_0):(X_1\rightrightarrows X_0)\to (X_1/G_1\rightrightarrows X_0/G_0)$ becomes a Riemannian groupoid submersion.
\end{proposition}
\begin{proof}
We carry out the proof for $1$-metrics on $X_1$ since the computations are analogous if we consider instead $2$-metrics on $X_2$. We start by exhibiting the Lie groupoid structure on $X_1/G_1\rightrightarrows X_0/G_0$. This fact was stated for instance in \cite[Prop. 3.6]{GZ} but without proof so that we exhibit a proof of it by sake of completeness. It is well known that $X_j/G_j$ admits a unique manifold structure so that $\pi_j$ (for $j=0,1$) is a surjective submersion. We define source and target maps respectively as $\overline{s}([p])=[s_X(p)]$ and $\overline{t}([p])=[t_X(p)]$ for all $p\in X_1$. As consequence of Identities \eqref{Equivariant} these maps are well defined and, moreover, both of them are surjective submersions since $\pi_j$, $s$, and $t$ are so. We have that $([p],[q])\in (X_1/G_1)_2$ if and only if $\overline{s}([p])=\overline{t}([q])$ which in turn holds true if and only if $s_X(p)=g_0t_X(q)$ for some $g_0\in G_0$. Thus, we define $\overline{m}([p],[q])=[m_X(p,gq)]$ for some $g\in G_1$ such that $t_G(g)=g_0$. It is simple to check that $\overline{m}$ does not depend on the choice of $g$ and it is well defined because of Property \eqref{MultAction}. This is also clearly smooth and associative since $m_X$ is associative and Identity \eqref{MultAction} is satisfied. The unit map and the inversion are respectively defined by $\overline{u}([x])=[u_X(x)]$ and $\overline{i}([p])=[i_X(p)]$ for all $x\in X_0$ and $p\in X_1$. It is also easy to verify that these maps are well defined, smooth, and they satisfy the required groupoid conditions.
	
	Consider, for $j=0,1$, the induced Riemannian metric $\overline{\eta}^{(j)}=(\pi_j)_\ast \eta^{(j)}$ on $X_j/G_j$ making of $\pi_j$ a Riemannian submersion. Moreover, note that
	$$\overline{i}_\ast \overline{\eta}^{(1)}=(\overline{i}\circ \pi_1)_\ast \eta^{(1)}=(\pi_1\circ i_X)_\ast\eta^{(1)}=(\pi_1)_\ast \eta^{(1)}=\overline{\eta}^{(1)},$$
	since $i$ is an isometry, and
	$$\overline{s}_\ast \overline{\eta}^{(1)}=(\overline{s}\circ \pi_1)_\ast \eta^{(1)}=(\pi_0\circ s_X)_\ast\eta^{(1)}=(\pi_0)_\ast \eta^{(0)}=(\pi_0\circ t_X)_\ast\eta^{(1)}=(\overline{t}\circ \pi_1)_\ast \eta^{(1)}=\overline{t}_\ast \overline{\eta}^{(1)},$$
	so that $\overline{s}_\ast \overline{\eta}^{(1)}=\overline{t}_\ast \overline{\eta}^{(1)}=\overline{\eta}^{(0)}$. Therefore, $(X_1/G_1\rightrightarrows X_0/G_0,\overline{\eta})$ with $\overline{\eta}=(\overline{\eta}^{(1)},\overline{\eta}^{(0)})$ is a Riemannian groupoid and $\pi=(\pi_1,\pi_0)$ is a Riemannian submersion of groupoids, as desired.
\end{proof}

Secondly, we can obtain $2$-equivariant weak linearizations around $G_0$-invariant saturated submanifolds in $X_0$. Namely, let $\theta$ be a $2$-action of $G_1\rightrightarrows G_0$ on $X_1\rightrightarrows X_0$ and let $S\subset X_0$ be a $G_0$-invariant saturated submanifold. We say that $X_1\rightrightarrows X_0$ is \emph{$2$-equivariant weakly linearizable} at $S$ if there are $G$-invariant Lie groupoid neighborhoods $\widetilde{V}\rightrightarrows V$ of $X_{S}\rightrightarrows S$ in $\nu(X_S)\rightrightarrows \nu(S)$ (seen as the zero section) and $\widetilde{U}\rightrightarrows U$ of $X_S\rightrightarrows S$ in $X_1\rightrightarrows X_0$, and a $2$-equivariant Lie groupoid isomorphism $\phi: (\widetilde{V}\rightrightarrows V)\xrightarrow[]{\cong} (\widetilde{U}\rightrightarrows U)$ which is the identity on $X_S\rightrightarrows S$.

The $G$-invariant property of the Lie groupoid neighborhood in $\nu(X_S)\rightrightarrows \nu(S)$ used above makes sense because of the following facts. Let us assume this time that we are given with a $2$-metric $\eta^{(2)}$ on $X_2$ and that $G_2$ acts on $(X_2,\eta^{(2)})$ by isometries. It is simple to check that every $2$-action $\theta=(\theta^1,\theta^0)$ of $G_1\rightrightarrows G_0$ on $X_1\rightrightarrows X_0$ induces a $2$-action $T\theta=(T\theta^1,T\theta^0)$ of $G_1\rightrightarrows G_0$ on $TX_1 \rightrightarrows TX_0$ by differentiating the actions $\theta^1$ and $\theta^0$. Let us pick a $G_0$-invariant saturated submanifold $S$ in $X_0$. This implies that $X_{S}$ is $G_1$-invariant and that $(X_{S})_2$ is $G_2$-invariant. If we respectively use $\eta^{(2)}$, $\eta^{(1)}$, and $\eta^{(0)}$ to identify $\nu(X_{S})_{2}\cong \nu((X_{S})_2)$ with $(T(X_{S})_2)^\perp$,  $\nu(X_{S})$ with $TX_S^\perp$, and $\nu(S)$ with $TS^\perp$ then it follows that the $2$-action $T\theta$ restrict to a well defined $2$-action $\overline{T\theta}$ of $G_1\rightrightarrows G_0$ on $\nu(X_{S})\rightrightarrows \nu(S)$ since $\theta$ is  isometric. Furthermore, the latter fact also implies that the exponential maps $\textnormal{exp}^{(2)}$, $\textnormal{exp}^{(1)}$, and $\textnormal{exp}^{(0)}$ are equivariant local diffeomorphisms. Therefore, as isometries preserve (horizontal) geodesics then by following the classical proofs of the equivariant tubular neighborhood theorem in \cite[VI. Thm. 2.2]{B} and \cite[Thm. 4.4]{IK,K} with the proof of the weak linearization theorem given in \cite[Thm. 5.11]{dHF} we easily obtain:

\begin{proposition}[2-equivariant weak groupoid linearization]\label{Lin1}
	Let $G_1\rightrightarrows G_0$ be a Lie $2$-group acting by isometries on a Riemannian groupoid $(X_1\rightrightarrows X_0,\eta)$. Then there exists a $2$-equivariant weak linearization of $X_1\rightrightarrows X_0$ around any $G_0$-invariant saturated submanifold $S$ in $X_0$.
\end{proposition}

Let us now deal with the classical approach used to ensure the existence of invariant Riemannian metrics on $G$-manifolds to provide a similar construction in our context. 

\begin{remark}
Let us denote by $\eta^\ast$ the dual metric associated to a Riemannian metric $\eta$. From \cite[Prop. 2.2]{dHF} we know that if $\pi:E\to B$ is a submersion and $\lbrace \eta_1,\cdots,\eta_k\rbrace$ is a collection of $\pi$-transverse metrics then its tangent average $\frac{1}{k}\sum_{l=1}^k\eta_l$ fails to be $\pi$-transverse again in general. Nevertheless, its cotangent average $\frac{1}{k}\left( \sum_{l=1}^k\eta_l^\ast\right)^\ast$ is always $\pi$-transverse which sometimes makes more advantageous to take a cotangent space point of view in the study of Riemannian submersions.
\end{remark}

%\begin{lemma}\label{Existence1}
%Let $G$ be a compact Lie group seen as a Lie unit 2-group $G\rightrightarrows G$ and suppose that acting on $(X_1\rightrightarrows X_0,\eta)$, then there exists a $1$-metric $\overline{\eta}$ making of such an action isometric. In particular, if $G$ is compact and $X_1\rightrightarrows X_0$ is proper then there always exists a $1$-metric making of such a $2$-action isometric.
%\end{lemma}
%\begin{proof}
%Let us consider the normalized Haar measure $\mu$ on $G$. We define $\overline{\eta}=(\overline{\eta}^{(1)},\overline{\eta}^{(0)})$ by averaging its dual as follows:
%$$(\overline{\eta}^{(1)})^\ast:=\int_{G}(\theta^1)^\ast_g (\eta^{(1)})^\ast d\mu(g).$$
%The dual metric $(\overline{\eta}^{(0)})^\ast$ has a similar defining formula but using $\eta^{(0)}$ and $\theta^0$ instead of $\eta^{(1)}$ and $\theta^1$. It is simple to check that $\overline{\eta}=(\overline{\eta}^{(1)},\overline{\eta}^{(0)})$ is the $1$-metric we are looking for. Last part of the assertion above follows from the fact that every proper Lie groupoid admits $n$-metrics \cite{dHF}.
%\end{proof}

Assume for a moment that $G_1$ is compact so that $G_0$ is also compact. If $\mu_1$ is the normalized Haar measure on $G_1$ then, by uniqueness, the pushforward measure $s_{G\ast}\mu_1$ agrees with the normalized Haar measure $\mu_0$ on $G_0$ since $s_G$ a surjective Lie group homomorphism and
\begin{equation}\label{PushHaar}
\int_{G_0}f d(s_{G\ast}\mu_1)=\int_{G_1} f\circ s_G d\mu_1,
\end{equation}
for each continuous function $f:G_0\to \mathbb{R}$. Analogously, $t_{G\ast}\mu_1=\mu_0$ and $i_{G\ast}\mu_1=\mu_1$. Note also that the fact that $i_{G\ast}\mu_1=\mu_1$ and $s_G\circ i_G=t_G$ immediately implies that $t_{G\ast}\mu_1=s_{G\ast}\mu_1$. Thus, we are in conditions to state:

\begin{theorem}\label{Existence2}
Suppose that $G_1\rightrightarrows G_0$ is a Lie $2$-group acting on a Riemannian groupoid $(X_1\rightrightarrows X_0,\eta)$. If $G_1$ compact then there exists a groupoid metric $\overline{\eta}$ on $X_1\rightrightarrows X_0$ for which the 2-action $\theta$ becomes isometric.
\end{theorem}
\begin{proof}

To simplify the computations we shall suppose that $\eta=\eta^{(1)}$ is a $1$-metric on $X_1$. By dual averaging we define:
$$(\overline{\eta}^{(1)})^\ast:=\int_{G_1}(\theta^1)^\ast_g(\eta^{(1)})^\ast d\mu_1(g)\quad\textnormal{and}\quad (\overline{\eta}^{(0)})^\ast:=\int_{G_0}(\theta^0)^\ast_{g_0}(\eta^{(0)})^\ast d\mu_0(g_0).$$
The result will follow from the fact that $\theta$ is a $2$-action, $\eta$ is already a $1$-metric, and Identity \eqref{PushHaar} holds true. Indeed, on the one hand, a straightforward computation using the fact that $i$ is an isometry of $(X_1,\eta^{(1)})$ verifying both Identity \eqref{Equivariant} and  $i_{G\ast}\mu_1=\mu_1$ implies that $i_X$ is an isometry of $(X_1,\overline{\eta}^{(1)})$ as well. The fact that $s_X:X_1\to X_0$ is a Riemannian submersion tells us that $ds_X(p)^\ast: T_{s_X(p)}^\ast X_0 \to \textnormal{ker}(ds_X(p))^\circ$ is an isometry for all $p\in X_1$ where $\textnormal{ker}(ds_X(p))^\circ$ denotes the annihilator of the vectors tangent to the fiber. Given $p\in X_1$, covectors  $\alpha,\beta\in T_{s_X(p)}^\ast X_0$, and $g\in G_1$ such that $s_G(g)=g_0$ we get the following chain of equalities:

% On the other hand, it is simple to check that if $\alpha \in \textnormal{ker}(ds_X(p))^{\perp_{\overline{\eta}^{(1)}}}$ then $d(\theta^1_g)_p(v)\in\textnormal{ker}(ds_X(gp))^{\perp_{\eta^{(1)}}}$ for all $g\in G_1$. Let $v,w\in \textnormal{ker}(ds_X(p))^{\perp_{\overline{\eta}^{(1)}}}$ and $g_0\in G_0$.

\begin{eqnarray*}
(\overline{\eta}^{(0)}_{s_X(p)})^\ast (\alpha,\beta) &=& \int_{G_0}(\eta^{(0)}_{g_0s_X(p)})^\ast ((\theta^0_{g_0})^\ast(\alpha ),(\theta^0_{g_0})^\ast(\beta ))d\mu_0\\
\eqref{PushHaar}&=& \int_{G_1}(\eta^{(0)}_{s_G(g)s_X(p)})^\ast ((\theta^0_{s_G(g)})^\ast(\alpha ),(\theta^0_{s_G(g)})^\ast(\beta ))d\mu_1\\
& = & \int_{G_1}(\eta^{(1)}_{gp})^\ast (ds_X(gp)^\ast((\theta^0_{s_G(g)})^\ast(\alpha )),ds_X(gp)^\ast((\theta^0_{s_G(g)})^\ast(\beta )))d\mu_1\\
& = & \int_{G_1}(\eta^{(1)}_{gp})^\ast ((\theta^1_{g})^\ast(ds_X(p)^\ast(\alpha )),(\theta^1_{g})^\ast(ds_X(p)^\ast(\beta )))d\mu_1\\
&= & (\overline{\eta}^{(1)}_{p})^\ast (ds_X(p)^\ast(\alpha ),ds_X(p)^\ast(\beta )),
\end{eqnarray*}
from which we conclude that $s_X$ is also Riemannian for the averaged metrics. With analogous computations we get a similar conclusion for $t_X:X_1\to X_0$ so that the result follows as claimed.
\end{proof}

\begin{remark}\label{Rmk4}
It is worth noticing that if we consider an $n$-metric on $X_n$ for $n\geq 2$ and take into account the simplicial approach described in Remark \ref{Rmk2} then by applying similar averaging arguments as in the proof of Theorem \ref{Existence2} it is possible to show the existence of another $n$-metric on $X_n$ that is invariant by the action of $G_n$ on $X_n$. For instance, let $\eta^{(2)}$ be a $2$-metric on $X_2$ and let $\mu_2$ denote the normalized Haar measure on $G_2$. On the one hand, by dual averaging we define $\displaystyle (\overline{\eta}^{(2)})^\ast:=\int_{G_2}(\theta^2)^\ast_{(g,h)}(\eta^{(2)})^\ast d\mu_2(g,h)$. On the other hand, it is simple to see that we can obtain analogous formulas as that in Identity \eqref{PushHaar} by using instead $\pi_{2,G}, m_G, \pi_{1,G}:G_2\to G_1$ as well as the Lie group isomorphisms $G_2\to G_2$ determined by the canonical action of $S_3$ on $G_2$. Hence, by similar computations as those in the main part of Theorem \ref{Existence2} it follows that $\overline{\eta}^{(2)}$ is a $2$-metric on $X_2$ for which the action of $G_2$ on $(X_2,\overline{\eta}^{(2)})$ becomes isometric.
\end{remark}

It is known that every proper Lie groupoid can be equipped with a groupoid metric, see \cite[Thm. 4.13]{dHF}. Hence, we get that:
\begin{corollary}
Let $\theta$ be a $2$-action of a Lie $2$-group $G_1\rightrightarrows G_0$ on a proper groupoid $X_1\rightrightarrows X_0$. If $G_1$ is compact then there always exists a Riemannian groupoid metric on $X_1\rightrightarrows X_0$ for which $\theta$ becomes an isometric $2$-action.
\end{corollary}

As an interesting application of Theorem \ref{Existence2} and Proposition \ref{PQuotient} we can prove analogous statements to those of the Slice Theorem and the Equivariant Tubular Neighborhood Theorem in our setting. Let us start by introducing the following terminology. If $\theta^0:G\times M\to M$ is a smooth action of a Lie group $G$ on a smooth manifold $M$ then a \emph{standard slice} at $x_0\in M$ is an embedded submanifold $S_{x_0}$ containing $x_0$ and satisfying the following properties:
\begin{enumerate}
\item[a.] $T_{x_0}M=T_{x_0} (G\cdot x_0) \oplus T_{x_0} S_{x_0}$ and $T_xM=T_{x} (G\cdot x)+T_{x} S_{x_0}$ for all $x\in S_{x_0}$,
\item[b.] $S_{x_0}$ is $\textnormal{Iso}_{G}(x_0)$-invariant, and 
\item[c.] if $x\in S_{x_0}$ and $g\in G$ are such that $\theta_g(x)\in S_{x_0}$ then $g\in \textnormal{Iso}_{G}(x_0)$.
\end{enumerate}
Here $G\cdot x_0$ and $\textnormal{Iso}_{G}(x_0)$ respectively denote the $G$-orbit and the $G$-isotropy group at $x_0$ with respect to $\theta$.

\begin{definition}\label{DefSlice}
Let $\theta$ be a $2$-action of $G_1\rightrightarrows G_0$ on $X_1\rightrightarrows X_0$. A \emph{groupoid slice} at $x_0\in X_0$ is defined to be a Lie subgroupoid $S_{1_{x_0}}\rightrightarrows S_{x_0}$ of $X_1\rightrightarrows X_0$ such that $S_{x_0}$ and $S_{1_{x_0}}$ are standard slices at $x_0$ and $1_{x_0}$, respectively.
\end{definition}
The following is a straightforward result.
\begin{lemma}\label{IsoOrbit}
Take any $x_0\in X_0$. There are natural structures of:
\begin{itemize}
\item Lie 2-subgroup between the $G$-isotropy groups $\textnormal{Iso}_{G_1}(1_{x_0})\rightrightarrows \textnormal{Iso}_{G_0}(x_0)$, and
\item Lie subgroupoid between the $G$-orbits $G_1\cdot 1_{x_0}\rightrightarrows G_0\cdot x_0$.
\end{itemize}
\end{lemma}

Our version of the Slice Theorem is as follows.

\begin{proposition}[Groupoid slice]\label{SliceThm}
Let $\theta$ be a proper $2$-action of a Lie $2$-group $G_1\rightrightarrows G_0$ on a Riemannian groupoid $(X_1\rightrightarrows X_0,\eta)$. Then there exists a groupoid slice $S_{1_{x_0}}\rightrightarrows S_{x_0}$ at each $x_0\in X_0$.

%with its $G_0$-orbit a (groupoid) saturated submanifold in $X_0$. \fcomment{Is the saturated requirement necessary?}
\end{proposition}
\begin{proof}
Consider the induced $2$-action of $\textnormal{Iso}_{G_1}(1_{x_0})\rightrightarrows \textnormal{Iso}_{G_0}(x_0)$ on $X_1\rightrightarrows X_0$. By applying Theorem \ref{Existence2} together with Remark \ref{Rmk4} we may use the $2$-metric $\eta$ on $X_2$ and use it to construct another $2$-metric $\tilde{\eta}$ on $X_2$ in such a way $\textnormal{Iso}_{G_1}(1_{x_0})\rightrightarrows \textnormal{Iso}_{G_0}(x_0)$ acts isometrically on $(X_1\rightrightarrows X_0,\tilde{\eta})$. As in the classical case \cite[Thm. 3.49]{AB}, we define $S_{x_0}$ by setting $S_{x_0}=\tilde{\exp}^{(0)}_{x_0}(B_\epsilon(0))$ where $B_\epsilon(0)$ is an open ball of radius $\epsilon>0$ around the origin in the normal space $\nu_{x_0} (G_0 \cdot x_0)$ to the $G_0$-orbit through $x_0$ (normal domain).

As $G_1\cdot 1_{x_0}\rightrightarrows G_0\cdot x_0$ is a Lie subgroupoid of $X_1\rightrightarrows X_0$ we have a well defined Lie subgroupoid $\nu(G_1\cdot 1_{x_0})\rightrightarrows \nu(G_0\cdot x_0)$ of $TX_1\rightrightarrows TX_0$ so that we may also consider the Lie groupoid $V_{B_\epsilon(0)}\rightrightarrows B_\epsilon(0)$ where $V_{B_\epsilon(0)}=\overline{ds}_{1_{x_0}}^{-1}(B_\epsilon(0))\cap\overline{dt}_{1_{x_0}}^{-1}(B_\epsilon(0))$. Let us use the groupoid metric $\tilde{\eta}$ to identify $\nu(G_1\cdot 1_{x_0})\cong T(G_1\cdot 1_{x_0})^\perp$ and $\nu(G_0\cdot x_0)\cong T(G_0\cdot x_0)^\perp$. By shrinking $B_\epsilon(0)$ if necessary we may assume that $V_{B_\epsilon(0)}$ is an open ball around the origin in the normal space $T_{1_{x_0}} (G_1\cdot 1_{x_0})^\perp$ to the $G_1$-orbit through $1_{x_0}$ on which $\tilde{\exp}^{(1)}_{1_{x_0}}$ is well defined. Therefore, we now set $S_{1_{x_0}}= \tilde{\exp}^{(1)}_{1_{x_0}}(V_{B_\epsilon(0)})$. Hence, by arguing with similar arguments as those used to prove the multiplicative property of the exponential maps associated to a Riemannian $2$-metric in \cite[Thm. 5.11]{dHF} together with the equivariant property these exponential maps have, we conclude that $S_{1_{x_0}}\rightrightarrows S_{x_0}$ is the Lie subgroupoid of $X_1\rightrightarrows X_0$ we are looking for.
\end{proof}

The previous proposition will be used in order to obtain a $2$-equivariant Tubular Neighborhood Theorem for the $G$-orbit groupoid $G_1\cdot 1_{x_0}\rightrightarrows G_0\cdot x_0$. In addition, we need to consider the construction of the associated groupoid bundle of a groupoid principal $2$-bundle. A particular case of such a construction can be found in \cite[Lem. 9.1.2]{hsz}. 

\begin{remark}\label{AssociatedBundle}
Let $\pi: (P_1\rightrightarrows P_0)\to (X_1\rightrightarrows X_0)$ be a  groupoid principal $2$-bundle with structural Lie $2$-group $G_1\rightrightarrows G_0$ \cite{CCK}. Assume that there exists a left $2$-action of $G_1\rightrightarrows G_0$ over another Lie groupoid $F_1\rightrightarrows F_0$. Given this data we can construct two associated fiber bundles $E_j:= P_j\times_{G_j}F_j$ over $X_j$ for $j=0,1$. These are defined as the quotient spaces $(P_j\times F_j)/G_j$ with respect to the actions $g_j\cdot (p_j,f_j)=(p_jg_j^{-1},g_jf_j)$, for all $g_j\in G_j$, $p_j\in P_j$ and $f_j\in F_j$, together with projections $\overline{\pi}_j([p_j,f_j])=\pi_j(p_j)$ onto $X_j$. It is simple to check that there exists a natural Lie groupoid structure $E_1\rightrightarrows E_0$ for which the projection $\overline{\pi}: (E_1\rightrightarrows E_0)\to (X_1\rightrightarrows X_0)$ becomes a Lie groupoid fibration. The source and target maps are the obvious ones, namely:
$$s_E([p,f])=[s_P(p),s_F(f)]\qquad\textnormal{and}\qquad t_E([p,f])=[t_P(p),t_F(f)],$$
and the groupoid composition is defined as follows. If $s_E([p,f])=t_E([q,l])$ then there is $g_0\in G_0$ such that $(s_P(p)g_0^{-1},g_0s_F(f))=(t_P(q),t_F(l))$. So, we set
$$m_E([p,f],[q,l])=[(pg^{-1})\ast q,(gf)\ast l],$$
for some $g$ inside the $s_G$-fiber at $g_0$. From Formula \eqref{MultAction} it follows that the latter equality is well defined and that it does not depend on the choice of $g\in s_G^{-1}(g_0)$. Furthermore, the same formula and the associativity of $m_X$ together imply the associativity of the composition $m_E$. As expected, the inversion $i_E$ and the unit map $u_E$ are defined in the obvious way.
\end{remark}

Let $\theta$ be a $2$-action of $G_1\rightrightarrows G_0$ on $X_1\rightrightarrows X_0$. From Lemma \ref{IsoOrbit}, Proposition  \ref{SliceThm}, the quotient construction described in Proposition \ref{PQuotient}, and Remark \ref{AssociatedBundle} we easily deduce that:
\begin{lemma}
There exists a principal groupoid $2$-bundle $\pi:(G_1\rightrightarrows G_0)\to (G_1/\textnormal{Iso}_{G_1}(1_{x_0})\rightrightarrows G_0/\textnormal{Iso}_{G_0}(x_0))$ with structural Lie $2$-group $\textnormal{Iso}_{G_1}(1_{x_0})\rightrightarrows \textnormal{Iso}_{G_0}(x_0)$, yielding an associated groupoid fibration $$(G_1\times_{\textnormal{Iso}_{G_1}(1_{x_0})}S_{1_{x_0}}\rightrightarrows G_0\times_{\textnormal{Iso}_{G_0}(x_0)}S_{x_0})\to (G_1/\textnormal{Iso}_{G_1}(1_{x_0})\rightrightarrows G_0/\textnormal{Iso}_{G_0}(x_0)),$$ 
with groupoid fiber $S_{1_{x_0}}\rightrightarrows S_{x_0}$.
\end{lemma}

Consider the classical $G$-invariant tubular neighborhoods $\textnormal{Tub}_{\epsilon}^{\tilde{\eta}}(G_1\cdot 1_{x_0})=\theta^1(G_1,S_{1_{x_0}})$ and $\textnormal{Tub}_{\epsilon}^{\tilde{\eta}}(G_0\cdot x_0)=\theta^0(G_0,S_{x_0})$ of $G_1\cdot 1_{x_0}$ and $G_0\cdot x_0$, respectively. Note that we have emphasized the inclusion of $\tilde{\eta}$ and $\epsilon$, from the proof of Proposition \ref{SliceThm}, at the definition of the above tubular neighborhoods since the existence of our groupoid slice relies on them. It is clear that $\textnormal{Tub}_{\epsilon}^{\tilde{\eta}}(G_1\cdot 1_{x_0})\rightrightarrows \textnormal{Tub}_{\epsilon}^{\tilde{\eta}}(G_0\cdot x_0)$ is a Lie subgroupoid of $X_1\rightrightarrows X_0$ and that the $2$-action $\theta$ restricts well over it. Furthermore, $\textnormal{Tub}_{\epsilon}^{\tilde{\eta}}(G_1\cdot 1_{x_0})\rightrightarrows \textnormal{Tub}_{\epsilon}^{\tilde{\eta}}(G_0\cdot x_0)$ determines an open Lie groupoid neighborhood of the $G$-orbit groupoid $G_1\cdot 1_{x_0}\rightrightarrows G_0\cdot x_0$. 

We are now in conditions to state our version of the Equivariant Tubular Neighborhood Theorem. Namely: 
\begin{theorem}[Equivariant groupoid tubular neighborhood]\label{ETubular}
Suppose that $\theta$ is a proper $2$-action of a Lie $2$-group $G_1\rightrightarrows G_0$ on a Riemannian groupoid $(X_1\rightrightarrows X_0,\eta)$. Then, for every $x_0\in X_0$ there exists a $2$-equivariant Lie groupoid isomorphism
$$\Psi: (G_1\times_{\textnormal{Iso}_{G_1}(1_{x_0})}S_{1_{x_0}}\rightrightarrows G_0\times_{\textnormal{Iso}_{G_0}(x_0)}S_{x_0}) \xrightarrow[]{\cong} (\textnormal{Tub}_{\epsilon}^{\tilde{\eta}}(G_1\cdot 1_{x_0})\rightrightarrows \textnormal{Tub}_{\epsilon}^{\tilde{\eta}}(G_0\cdot x_0)).$$
\end{theorem}
\begin{proof}
The left $2$-action of $G_1\rightrightarrows G_0$ on $G_1\times_{\textnormal{Iso}_{G_1}(1_{x_0})}S_{1_{x_0}}\rightrightarrows G_0\times_{\textnormal{Iso}_{G_0}(x_0)}S_{x_0}$ we will consider here is the one given by $\overline{g_j}\cdot [g_j,f_j]=[\overline{g_j}g_j,f_j]$ for all $\overline{g_j},g_j\in G_j$ and $f_j\in S_j$. Here $S_1=S_{1_{x_0}}$ and $S_0=S_{x_0}$.
The Lie groupoid isomorphism $\Psi$ is defined as 
$$\Psi^j([g_j,f_j])=\theta^j(g_j,f_j).$$
From the proof of Theorem 3.57 in \cite{AB} we already know that $\Psi^j$ is a $G_j$-equivariant diffeomorphism. Thus, we only have to check that $\Psi$ defines indeed a Lie groupoid morphism. By using the structural maps defined in Remark \ref{AssociatedBundle} we have that
$$(\Psi^0\circ s_E)([g,f])=\Psi^0([s_G(g),s_X(f)])=s_G(g)s_X(f)=s_X(gf)=(s_X\circ \Psi^1)([g,f]).$$
We can similarly obtain that $\Psi^0\circ t_E=t_X\circ \Psi^1$. Moreover, by applying Formula \eqref{MultAction} we get
\begin{eqnarray*}
\Psi^1([g,f]\ast [h,l]) &=& \Psi^1([(g\overline{g}^{-1})\ast h,(\overline{g}f)\ast l])=((g\overline{g}^{-1})\ast h)\cdot((\overline{g}f)\ast l)\\
&=& (gf)\ast (hl)=\Psi^1([g,f])\ast \Psi^1([h,l]).
\end{eqnarray*}
Hence, the result follows as desired.
\end{proof}
We finish the discussion of this subsection by briefly commenting that it is possible to define a groupoid $G$-orbit type topological stratification for $X_1\rightrightarrows X_0$. We say that $x_0$ and $y_0$ in $X_0$ have the same $G$-\emph{orbit type} if there exists a Lie groupoid isomorphism $\Phi$ between $G_1\cdot 1_{x_0}\rightrightarrows G_0\cdot x_0$ and $G_1\cdot 1_{y_0}\rightrightarrows G_0\cdot y_0$ such that $\Phi^1$ is $G_1$-equivariant. This automatically implies that $\Phi^0$ is $G_0$-equivariant. It is clear that this $G$-orbit type requirement defines an equivalent relation $\sim$  on $X_0$ for which we denote by $M_{x_0}^\sim$ the equivalent class at $x_0\in X_0$ associated to such a relation. We claim that $M_{x_0}^\sim$ is saturated in $X_0$, that is, $s_X^{-1}(M_{x_0}^\sim)=t_X^{-1}(M_{x_0}^\sim)$. If $p\in s_X^{-1}(M_{x_0}^\sim)$ then $s_X(p)\sim x_0$ so that there is a $(G_1\rightrightarrows G_0)$-equivariant isomorphism between $G_1\cdot 1_{s_X(p)}\rightrightarrows G_0\cdot s_X(p)$ and $G_1\cdot 1_{x_0}\rightrightarrows G_0\cdot x_0$. It is simple to check that $\Phi_p$ defined by $\Phi_p^1(g1_{t_X(p)})=g1_{s_X(p)}$ and $\Phi_p^0(g_0t_X(p))=g_0s_X(p)$ defines another $(G_1\rightrightarrows G_0)$-equivariant isomorphism between $G_1\cdot 1_{t_X(p)}\rightrightarrows G_0\cdot t_X(p)$ and $G_1\cdot 1_{s_X(p)}\rightrightarrows G_0\cdot s_X(p)$ so that by taking $\Phi\circ \Phi_p$ we conclude that $t_X(p)\sim x_0$. The other inclusion may be similarly verified. Thus, by setting $M_{1_{x_{0}}}^\sim=s_X^{-1}(M_{x_0}^\sim)=t_X^{-1}(M_{x_0}^\sim)$ we obtain a collection of topological groupoids $\lbrace M_{1_{x_{0}}}^\sim \rightrightarrows M_{x_0}^\sim\rbrace_{x_0\in X_0}$ which somehow stratifies the Lie groupoid $X_1\rightrightarrows X_0$. The previous observation suggests that by combining classical ideas from \cite[s. 3.5]{AB} with some of the results obtained in this section it could be possible to show that each of the $M_{1_{x_{0}}}^\sim \rightrightarrows M_{x_0}^\sim$ is a honest Lie subgroupoid. Nevertheless, the local $G$-orbit type notion from the classical case does not extended directly in our case. \emph{We conjecture that our guess is true}. That is, $M_{1_{x_{0}}}^\sim \rightrightarrows M_{x_0}^\sim$ is Lie subgroupoid of $X_1\rightrightarrows X_0$ for all $x_0\in X_0$.

\subsection{Orthogonal Lie 2-groups}\label{S:3:1}
In this subsection we derive another application of Theorem \ref{Existence2} which has to do with the construction of groupoid bi-invariant Riemannian metrics on compact Lie $2$-groups. To simplify computations in this subsection will be enough to consider only $1$-metrics. Let $G_1\rightrightarrows G_0$ be a Lie 2-group and consider the pairs $L=(L^1,L^0)$ and $R=(R^1,R^0)$ where $L^j$ and $R^j$ for $j=0,1$ are respectively the actions of $G_j$ on itself determined by left and right multiplications. It is simple to check that as consequence of Identity \eqref{MultiProduct} and the fact that the structural maps of $G_1\rightrightarrows G_0$ are Lie group homomorphisms it follows that $L$ and $R$ determine left $2$-actions of $G_1\rightrightarrows G_0$ on itself.
\begin{definition}
A Lie $2$-group $G_1\rightrightarrows G_0$ is said to be \emph{orthogonal} if it may be equipped with a $1$-metric for which both $L$ and $R$ are isometric $2$-actions. Such a $1$-metric will be called \emph{bi-invariant}.
\end{definition}
We will think of $1$-metrics on a Lie $2$-algebra $\mathfrak{g}_1\rightrightarrows \mathfrak{g}_0$ as pairs of inner products $\langle \cdot,\cdot\rangle=(\langle \cdot,\cdot\rangle^{(1)},\langle \cdot,\cdot\rangle^{(0)})$ verifying the required conditions of $1$-metric. Let $G_1\rightrightarrows G_0$ be a Lie 2-group with respective Lie 2-algebra $\mathfrak{g}_1\rightrightarrows \mathfrak{g}_0$. We denote by $\textnormal{Ad}=(\textnormal{Ad}^1,\textnormal{Ad}^0)$ the $2$-action of $G_1\rightrightarrows G_0$ on $\mathfrak{g}_1\rightrightarrows \mathfrak{g}_0$ determined by the adjoint actions $\textnormal{Ad}^j$ of $G_j$ on $\mathfrak{g}_j$ for $j=0,1$. This $2$-action will be called \emph{adjoint $2$-action} of $G_1\rightrightarrows G_0$. 

\begin{proposition}\label{Bi-invariant1}
A Lie $2$-group $G_1\rightrightarrows G_0$ is orthogonal if and only if there exists a $1$-metric $\langle \cdot,\cdot\rangle$ on its Lie $2$-algebra $\mathfrak{g}_1\rightrightarrows \mathfrak{g}_0$ for which the adjoint $2$-action is by linear isometries.
\end{proposition}
\begin{proof}
It is clear that if $\eta$ is a bi-invariant $1$-metric then $\eta_e=(\eta^1_{e_1},\eta^0_{e_0})$, where $e_j$ is the identity element in $G_j$, defines a $1$-metric on $\mathfrak{g}_1\rightrightarrows \mathfrak{g}_0$ for which the adjoint $2$-action is by linear isometries. Conversely, given such a $\langle \cdot,\cdot\rangle$, we define $\eta$ on $G_1\rightrightarrows G_0$ by setting
$$\eta^{(1)}_g(v,w)=\langle d(L^1_{g^{-1}})_g(v),d(L^1_{g^{-1}})_g(w)\rangle^{(1)}.$$
The metric $\eta^{(0)}$ has a similar defining formula but using $L^0$ and $\langle \cdot,\cdot\rangle^{(0)}$ instead of $L^1$ and $\langle \cdot,\cdot\rangle^{(1)}$. It is clear that $\eta^{(1)}$ is bi-invariant. Therefore, it remains to prove that $\eta$ defines indeed a $1$-metric. Firstly, for $g\in G_1$ and $v,w\in T_g G_1$ we have
\begin{eqnarray*}
i^\ast_g\eta^{(1)}(v,w) & = & \langle d(L^1_{i(g)^{-1}})_{i(g)}(di_g(v)),d(L^1_{i(g)^{-1}})_{i(g)}(di_g(w))\rangle^{(1)}\\
&=&  \langle d(L^1_{i(g^{-1})}\circ i)_g(v),d(L^1_{i(g^{-1})}\circ i)_g(w)\rangle^{(1)}\\
& = & \langle d(i\circ L^1_{g^{-1}})_g(v),d(i\circ L^1_{g^{-1}})_g(w)\rangle^{(1)}\\
&=&\langle di_{e_1}(d(L^1_{g^{-1}})_g(v)), di_{e_1}(d(L^1_{g^{-1}})_g(v))\rangle^{(1)}=\eta^{(1)}_g(v,w).
\end{eqnarray*}
Secondly, let $g_0\in G_0$ and $v,w\in T_{g_0}G_0$. It is clear that there are $\tilde{v},\tilde{w}\in \textnormal{ker}(ds(g))^{\perp_{\eta^{(1)}}}$ with $s(g)=g_0$ such that $ds_g(\tilde{v})=v$ and $ds_g(\tilde{w})=w$. Thus
\begin{eqnarray*}
(s_\ast \eta^{(1)})_{g_0}(v,w) & = & \eta^{(1)}_g(\tilde{v},\tilde{w}) = \langle d(L^1_{g^{-1}})_g(\tilde{v}),d(L^1_{g^{-1}})_g(\tilde{w})\rangle^{(1)}\\
& = & \langle ds_{e_1}(d(L^1_{g^{-1}})_g(\tilde{v})), ds_{e_1}(d(L^1_{g^{-1}})_g(\tilde{w}))\rangle^{(0)}\\
&=& \langle d(s\circ L^1_{g^{-1}})_g(\tilde{v}),d(i\circ L^1_{g^{-1}})_g(\tilde{w})\rangle^{(0)}\\
&=& \langle d(L^0_{g_0^{-1}}\circ s)_g(\tilde{v}),d(L^0_{g_0^{-1}}\circ s)_g(\tilde{w})\rangle^{(0)}\\
&=& \langle d(L^0_{g_0^{-1}})_{g_0}(v),d(L^0_{g_0^{-1}})_{g_0}(v)\rangle^{(0)}=\eta^{(0)}_{g_0}(v,w).
\end{eqnarray*}
Analogously, $t_\ast \eta^{(1)}=\eta^{(0)}$. So, the result follows.
\end{proof}
From now on we assume that the Lie groups we are working with are connected. It is well known that bi-invariant metrics on a Lie group are in one-to-one correspondence with inner products on its Lie algebra for which the adjoint representation determines infinitesimal isometries \cite{Me,Mi}. A Lie $2$-algebra $\mathfrak{g}_1\rightrightarrows \mathfrak{g}_0$ is said to be \emph{orthogonal} if it admits a $1$-metric $\langle \cdot,\cdot\rangle=(\langle \cdot,\cdot\rangle^{(1)},\langle \cdot,\cdot\rangle^{(0)})$ for which the adjoint representation $\textnormal{ad}^1:\mathfrak{g}_1\to \textnormal{Der}(\mathfrak{g}_1)$ acts by infinitesimal isometries on $(\mathfrak{g}_1,\langle \cdot,\cdot\rangle^{(1)})$. Therefore, as consequence of Proposition \ref{Bi-invariant1} and \cite[Lem. 7.2]{Mi} we get:
\begin{corollary}\label{MilnorCharacterization}
A Lie $2$-group is orthogonal if and only if its Lie algebra $\mathfrak{g}_1\rightrightarrows \mathfrak{g}_0$ is orthogonal.
\end{corollary}
From Theorem \ref{Existence2} we get that:
\begin{corollary}\label{Bi-Existence}
Every Lie $2$-group $G_1\rightrightarrows G_0$ with $G_1$ compact can be endowed with a bi-invariant $1$-metric.
\end{corollary}

Given an orthogonal Lie $2$-algebra $\mathfrak{g}_1\rightrightarrows \mathfrak{g}_0$ we may split $\mathfrak{g}_1$ as a direct sum of ideals $\mathfrak{g}_1=\mathfrak{h}\oplus \mathfrak{h}^{\perp_1}$ where $\mathfrak{h}=\textnormal{ker}(s)$. As the unit map $u$ is a canonical bisection we would expect that $u(x)\in \mathfrak{h}^{\perp_1}$ for all $x\in \mathfrak{g}_0$ but, however, this is not true in general unless we assume ``non-canonical" identifications. We say that an orthogonal Lie $2$-algebra is \emph{trivial} if $\mathfrak{h}^{\perp_1}=u(\mathfrak{g}_0)$. This notion comes up by the following simple result.

\begin{proposition}\label{0-orthogonal}
	If $\mathfrak{g}_1\rightrightarrows \mathfrak{g}_0$ is a trivial orthogonal Lie $2$-algebra then $\textnormal{ad}^1_{u(x)}|_\mathfrak{h}=0$ for all $x\in \mathfrak{g}_0$. In consequence, $\mathfrak{h}$ is abelian and $\textnormal{im}(t|_\mathfrak{h})\subseteq \mathfrak{z}(\mathfrak{g}_0)$.
\end{proposition}
\begin{proof}
	On the one hand, note that for any $y\in \mathfrak{h}$ and $z\in \mathfrak{g}$ one gets
	$$\langle \textnormal{ad}^1_{u(x)}(y),z\rangle^{(1)}=\langle u(x),\textnormal{ad}^1_{y}(z)\rangle^{(1)}=0,$$
	since $\mathfrak{h}$ is an ideal. Thus, as $\langle \cdot,\cdot\rangle^{(1)}$ is nondegenerate we get that $\textnormal{ad}^1_{u(x)}|_\mathfrak{h}=0$ for all $x\in \mathfrak{g}_0$. On the other hand, for all $y,y'\in \mathfrak{h}$ and $x\in \mathfrak{g}$ it follows that $[y,y']=\textnormal{ad}^1_{u(t(y))}(y')=0$ and $0=t([u(x),y])=[x,t(y)]$ since $t$ is a Lie algebra homomorphism.
\end{proof}
This in particular implies that the crossed module of Lie algebras associated to a trivial orthogonal Lie $2$-algebra is of the form $(\mathfrak{g},\mathfrak{h},\partial,0)$ with $\mathfrak{h}$ an abelian Lie algebra.
\begin{remark}
Suppose that we have a crossed module of Lie algebras $(\mathfrak{g},\mathfrak{h},\partial,\mathcal{L})$ which is $\mathcal{L}$-orthogonal in the sense of \cite{Ba,FMM}. That is, there exist inner products $\langle \cdot,\cdot\rangle_\mathfrak{g}$ and $\langle \cdot,\cdot\rangle_\mathfrak{h}$ such that the Lie algebra representation $\mathcal{L}:\mathfrak{g}\to\textnormal{Der}(\mathfrak{h})$ acts by infinitesimal isometries on $(\mathfrak{h},\langle \cdot,\cdot\rangle_\mathfrak{h})$ and the adjoint representation of $\mathfrak{g}$ acts by infinitesimal isometries on $(\mathfrak{g},\langle \cdot,\cdot\rangle_\mathfrak{g})$. Note that this directly implies that the adjoint representation of $\mathfrak{h}$ acts by infinitesimal isometries on $(\mathfrak{h},\langle \cdot,\cdot\rangle_\mathfrak{h})$. On the one hand, recall that the associated Lie $2$-algebra constructed with the crossed module data has $\mathfrak{g}_1=\mathfrak{h}\rtimes \mathfrak{g}$ with Lie algebra structure provided by the semi-direct product with respect to $\mathcal{L}$. Therefore, a straightforward computation allows us to conclude that the adjoint representation of $\mathfrak{g}_1$ acts by infinitesimal isometries with respect to $\langle \cdot,\cdot\rangle_\mathfrak{h}+\langle \cdot,\cdot\rangle_\mathfrak{g}$ if and only if $\mathcal{L}=0$. On the other hand, since the inversion $i:\mathfrak{g}_1\to \mathfrak{g}_1$ is given by $i(x,y)=(-x,y+\partial(x))$ for all $x,x'\in \mathfrak{h}$ and $y,y'\in \mathfrak{g}$ then it follows that $i$ is an isometry with respect to $\langle \cdot,\cdot\rangle_\mathfrak{h}+\langle \cdot,\cdot\rangle_\mathfrak{g}$ if and only if $\partial=0$. As consequence, there is no canonical correspondence between our notion of orthogonal Lie $2$-algebras and the notion of $\mathcal{L}$-orthogonal crossed module of Lie algebras which is known in the literature.
\end{remark}

\subsection{Some examples and constructions}\label{S:3:2}
It this short subsection we exhibit some examples and interesting constructions in which isometric Lie $2$-group actions naturally appear.

\begin{example}
Classical isometric actions of Lie groups $G$ on Riemannian manifolds $M$ are recovered from isometric $2$-actions of unit Lie $2$-groups $G \rightrightarrows G$ acting upon unit Riemannian groupoids $M \rightrightarrows M$.
\end{example}

\begin{example}
Let $(X_1\rightrightarrows X_0,\eta)$ be a Riemannian groupoid and $(\xi,v)$ be a complete multiplicative Killing vector field on $X_1\rightrightarrows X_0$. Here we consider $\xi$ a Killing vector field on $(X_1,\eta^{(1)})$ and $v$ a Killing vector field on $(X_0,\eta^{(0)})$. From \cite[Prop. 3.5]{MX} we know that the pair of flows defined by $(\xi,v)$ determine global automorphisms on $X_1\rightrightarrows X_0$ so that we get a well defined isometric $2$-action of $\mathbb{R}\rightrightarrows \mathbb{R}$ on $(X_1\rightrightarrows X_0,\eta)$. In particular, the flow of the multiplicative vector field $((1_\xi)_{X_1},\xi_{X_0})$ formed by the fundamental vector fields of an isometric $2$-action determines another isometric $2$-action.
\end{example}

\begin{example}\label{ExampleOrthogonal}
Let $G$ be an orthogonal Lie group and $H\leq G$ be a normal Lie subgroup. It is clear that $H$ acts on $G$ by left multiplication, leading to the action Lie groupoid $H\times G\rightrightarrows G$. Note that its space of arrows has a group structure, namely the semi-direct product by the
conjugation action $C_g(h)=ghg^{-1}$ of $G$ on $H$ so that we get a well defined Lie $2$-group. More importantly, by applying the gauge trick construction behind Proposition 4.7 and Example 4.9 in \cite{dHF} we conclude that it is possible to cook up a $1$-metric on $H\rtimes G\rightrightarrows G$ made out from the initial bi-invariant metric on $G$ in such a way it becomes an orthogonal Lie $2$-group.
\end{example}

%\begin{example}
%	Lie $2$-actions on proper étale groupoids?
%\end{example}

\begin{example}
Let $(M,\mathcal{F})$ be a regular Riemannian foliation and consider a free and proper isometric foliated action $G\times (M,\mathcal{F})\to (M,\mathcal{F})$ of a Lie group $G$. From \cite[Thm. 3.7]{GZ} it is known that the Lie $2$-group $G\rtimes G\rightrightarrows G$, as defined in Example \ref{ExampleOrthogonal}, determines a canonical $2$-action on the holonomy groupoid $\textnormal{Hol}(M,\mathcal{F})\rightrightarrows M$ which  extends the given action of $G$ on $M$. The Riemannian metric on $M$ completely determines a $0$-metric on $\textnormal{Hol}(M,\mathcal{F})\rightrightarrows M$ which can be extended to a $1$-metric \cite[Ex. 3.12]{dHF}. Thus, the fact that $G$ acts on $M$ isometrically implies that the extended $2$-action of $G\rtimes G\rightrightarrows G$ on $\textnormal{Hol}(M,\mathcal{F})\rightrightarrows M$ is by isometries.
\end{example}

\begin{example}
If $(M,\eta^{(0)})$ is a Riemannian manifold then $\eta^{(1)}=\eta^{(0)}\oplus \eta^{(0)}$ defines a $1$-metric on the pair groupoid $M\times M\rightrightarrows M$. Thus, if $G$ is a Lie group acting freely, properly and isometrically on $(M,\eta^{(0)})$ then the unit Lie $2$-group $G\rightrightarrows G$ acts isometrically on $M\times M\rightrightarrows M$ and in light of Proposition \ref{PQuotient} it follows that the quotient groupoid $(M\times M)/G\rightrightarrows M/G$ canonically inherits a $1$-metric. If $G$ admits a bi-invariant metric then the previous procedure yields an easy way to construct $1$-metrics on the gauge groupoid associated to a principal $G$-bundle over a Riemannian manifold (see Example \ref{ExPrincipalWarping} below).
\end{example}

In the next construction we will consider the notion of multiplicative $2$-connection $\omega=(\omega^1,\omega^0)$ on a groupoid principal $2$-bundle $\pi:(P_1\rightrightarrows P_0)\to (X_1\rightrightarrows X_0)$ with structural Lie $2$-group $G_1\rightrightarrows G_0$ as defined for instance in \cite[s. 5]{CCK}. 
\begin{example}[Principal groupoid warping]\label{ExPrincipalWarping}
Suppose that $X_1\rightrightarrows X_0$ is a Lie groupoid equipped with a $1$-metric and that $G_1\rightrightarrows G_0$ is an orthogonal Lie $2$-group. Let us prove that there exists a $1$-metric $\overline{\eta}$ on $P_1\rightrightarrows P_0$ for which the $2$-action of $G_1\rightrightarrows G_0$ is isometric and such that $\pi=(\pi_1,\pi_0):(P_1\rightrightarrows P_0)\to (X_1\rightrightarrows X_0)$ is a Riemannian groupoid submersion. Consider the associated $1$-metric $\langle \cdot,\cdot\rangle$ on the Lie $2$-algebra $\mathfrak{g}_1\rightrightarrows \mathfrak{g}_0$ of $G_1\rightrightarrows G_0$ and define 
	$$\overline{\eta}^{(1)}(v,w)=\eta^{(1)}(d\pi_1(v),d\pi_1(w))+\langle \omega^1(v),\omega^1(w)\rangle^{(1)}.$$
	The metric $\overline{\eta}^{(0)}$ is similarly defined by using instead $\eta^{(0)}$, $\pi_0$, $\langle \cdot,\cdot\rangle^{(0)}$, and $\omega^0$. It is simple to check that this expression yields a well defined right $G_j$-invariant metric on $P_j$ for which $\pi_j:P_j\to X_j$ becomes a Riemannian submersion ($j=0,1$). This is because $\pi_j$ is constant along the action orbits, $\omega^j$ is of $\textnormal{Ad}^j$-invariant type, and $\textnormal{ker}(d\pi_j)^{\perp_{\overline{\eta}^{(j)}}}=\textnormal{ker}(\omega^j)$ (compare \cite[p. 467]{O'n}). Let us verify that $\overline{\eta}=(\overline{\eta}^{(1)},\overline{\eta}^{(0)})$ determines a $1$-metric on $P_1\rightrightarrows P_0$. Note that, by abusing a little on the notation, we may rewrite $\overline{\eta}^{(j)}$ in a simpler way as $\overline{\eta}^{(j)}=(\pi_j)^\ast\eta^{(j)}+(\omega^j)^\ast \langle \cdot,\cdot\rangle^{(j)}$. Recall that $\pi$ and $\omega:(TP_1 \rightrightarrows TP_0)\to (\mathfrak{g}_1\times P_1\rightrightarrows \mathfrak{g}_0\times P_0)$ are Lie groupoid morphisms. So, on the one hand we get
	\begin{eqnarray*}
		(i_P)^\ast \overline{\eta}^{(1)} & = & (\pi_1\circ i_P)^\ast\eta^{(1)}+(\omega^1\circ i_P)^\ast \langle \cdot,\cdot\rangle^{(1)}=(i_X\circ \pi_1)^\ast\eta^{(1)}+(d(i_G)_{e_1}\circ \omega^1)^\ast \langle \cdot,\cdot\rangle^{(1)}\\
		&=& (\pi_1)^\ast\eta^{(1)}+(\omega^1)^\ast \langle \cdot,\cdot\rangle^{(1)}=\overline{\eta}^{(1)}.
	\end{eqnarray*}
	On the other hand, if $v\in \textnormal{ker}(d(s_P)_{p})^{\perp_{\overline{\eta}^{(1)}}}$ then it is simple to verify that the identities $\pi_0\circ s_P=s_X\circ \pi_1$ and $\omega^0\circ s_P=d(s_G)_{e_1}\circ \omega^1$ imply that $d(\pi_1)_p(v)\in \textnormal{ker}(d(s_X)_{\pi_1(p)})^{\perp_{\eta^{(1)}}}$ and $\omega^1(v)\in \textnormal{ker}(d(s_G)_{e_1})^{\perp_{\langle \cdot,\cdot\rangle^{(1)}}}$. Let us pick $v_1,v_2\in \textnormal{ker}(d(s_P)_{p})^{\perp_{\overline{\eta}^{(1)}}}$. Thus,
	\begin{eqnarray*}
		\overline{\eta}^{(0)}_{s_P(p)}(ds_P(p)(v_1),ds_P(p)(v_2)) & = & \eta^{(0)}_{\pi_0(s_P(p))}(d(\pi_0)_{s_P(p)}(ds_P(p)(v_1)),d(\pi_0)_{s_P(p)}(ds_P(p)(v_2)))\\
		& + & \langle \omega^0(ds_P(p)(v_1)),\omega^0(ds_P(p)(v_2))\rangle^{(0)}\\
		& = & \eta^{(0)}_{s_X(\pi_1(p))}(d(s_X)_{\pi_1(p)}(d\pi_1(p)(v_1)),d(s_X)_{\pi_1(p)}(d\pi_1(p)(v_1)))\\
		& + & \langle d(s_G)_{e_1}(\omega^1(v_1)),d(s_G)_{e_1}(\omega^1(v_2))\rangle^{(0)}\\
		& = & \eta^{(1)}_{\pi_1(p)}(d\pi_1(p)(v_1),d\pi_1(p)(v_1)) +  \langle \omega^1(v_1),\omega^1(v_2)\rangle^{(0)}\\
		& = & \overline{\eta}^{(1)}(v_1,v_2).
	\end{eqnarray*}
	Analogously, it follows that $(t_P)_\ast \overline{\eta}^{(1)}=\overline{\eta}^{(0)}$. Hence, we have shown that $(P_1\rightrightarrows P_0,\overline{\eta})$ is a Lie groupoid equipped with a $1$-metric $\overline{\eta}$ for which $G_1$ acts isometrically on $(P_1,\overline{\eta})$ and such that $\pi$ is a Riemannian groupoid submersion, as claimed.
	
	After performing similar computations it is simple to verify an analogous result for $2$-metrics instead. Namely, let us now suppose that $X_2$ can be endowed with a $2$-metric $\eta^{(2)}$ and that $G_2$ admits a bi-invariant $2$-metric with associated $2$-metric $\langle \cdot,\cdot\rangle^{(2)}$ on $\mathfrak{g}_2$. If $\omega^2:TP_2\to \mathfrak{g}_2\times P_2$ denotes the induced connection $1$-form by $\omega$ on the principal bundle $\pi_2:P_2\to X_2$ with structural group $G_2$ then the formula $\overline{\eta}^{(2)}((v,w),(v',w'))=\eta^{(2)}(d\pi_2(v,w),d\pi_2(v',w'))+\langle \omega^2(v,w),\omega^2(v',w')\rangle^{(2)}$ defines a $2$-metric on $P_2$ for which the action of $G_2$ on $(P_2,\overline{\eta}^{(2)})$ is by isometries and $\pi_2$ becomes a Riemannian submersion.
\end{example}

\begin{example}
Let $(X_1\rightrightarrows X_0,\eta)$ be a Riemannian étale groupoid with $n=\dim X_0=\dim X_1$. We denote by $\pi_1:O(X_1)\to X_1$ and $\pi_0:O(X_0)\to X_0$ to the principal $O(n,\mathbb{R})$-bundles of orthonormal frames over $(X_1,\eta^{(1)})$ and $(X_0,\eta^{(0)})$, respectively. There exists a canonical Lie groupoid structure $O(X_1)\rightrightarrows O(X_0)$ with the property that $\pi:(O(X_1)\rightrightarrows O(X_0))\to (X_1\rightrightarrows X_0)$ becomes a principal $2$-bundle with structural Lie $2$-group $O(n,\mathbb{R})\rightrightarrows O(n,\mathbb{R})$. Moreover, it admits a canonical multiplicative $2$-connection $\omega=(\omega^1,\omega^0)$, where $\omega^1$ and $\omega^0$ are respectively determined by the Levi-Civita connections on $(X_1,\eta^{(1)})$ and $(X_0,\eta^{(0)})$. The source and target maps of $O(X_1)\rightrightarrows O(X_0)$ are given by
$$\tilde{s}(v_1,\cdots,v_n)=(ds(v_1),\cdots,ds(v_n))\quad\textnormal{and}\quad \tilde{t}(v_1,\cdots,v_n)=(dt(v_1),\cdots,dt(v_n)),$$
and the composition is
$$\tilde{m}((v_1,\cdots,v_n),(v_1',\cdots,v_n'))=(dm(v_1,v_1'),\cdots,dm(v_n,v_n')).$$
The other structural maps can be defined in a similar fashion. Note that the principal warping construction from Example \ref{ExPrincipalWarping} applies in this case.
\end{example}
Next construction comes motivated by a beautiful notion known in the literature as \emph{Cheeger deformation}. The classical construction can be found for instance in \cite[s. 6.1]{AB}.
\begin{example}[Cheeger groupoid deformation]
	Suppose that $(G_1\rightrightarrows G_0, Q)$ is an orthogonal Lie 2-group, with $G_1$ compact, acting isometrically on a Riemannian groupoid $(X_1\rightrightarrows X_0,\eta)$. On the Lie groupoid product $X_1\times G_1\rightrightarrows X_0\times G_0$ we can consider the 1-metric $\eta\oplus \frac{1}{\tau}Q$.
	
	There is a natural free $2$-action of $G_1\rightrightarrows G_0$ on $X_1\times G_1\rightrightarrows X_0\times G_0$ where the $G_j$-actions, $j=1,0$, on $X_j\times G_j$ are given by
	
	\begin{equation}\label{CheegerD}
		h_j\cdot(p_j,g_j)=(h_jp_j,h_jg_j),\qquad p_j\in X_j,\ h_j,g_j\in G_j.
	\end{equation} 
	
	We claim that the quotient groupoid $\frac{X_1\times G_1}{G_1}\rightrightarrows\frac{X_0\times G_0}{G_0}$ determined by the previous actions is isomorphic to $X_1\rightrightarrows X_0$. Let us consider the groupoid principal $2$-bundle $G_1\rightrightarrows G_0$ over the point groupoid $\lbrace e_1\rbrace\rightrightarrows\lbrace e_0\rbrace$ with structural Lie $2$-group $G_1\rightrightarrows G_0$. By applying Remark \ref{AssociatedBundle} we know that by using the $2$-action of $G_1\rightrightarrows G_0$ on $X_1\rightrightarrows X_0$ we can construct the associated Lie groupoid bundle $(X_1\times_{G_1}G_1\rightrightarrows X_0\times_{G_0}G_0)\to (\lbrace e_1\rbrace\rightrightarrows\lbrace e_0\rbrace)$. It is important to notice that $X_1\times_{G_1}G_1\rightrightarrows X_0\times_{G_0}G_0$ is precisely the quotient groupoid $\frac{X_1\times G_1}{G_1}\rightrightarrows\frac{X_0\times G_0}{G_0}$. Therefore, as the new Lie groupoid bundle has groupoid fiber $X_1\rightrightarrows X_0$ and base groupoid $\lbrace e_1\rbrace\rightrightarrows\lbrace e_0\rbrace$ then we have the desired isomorphism. Under this identification the canonical groupoid projection $\pi=(X_1\times G_1\rightrightarrows X_0\times G_0)\to (X_1\rightrightarrows X_0)$ is formed by the maps $\pi_j(p_j,g_j)=g_j^{-1}p_j$.
	
	Observe that the $2$-action \eqref{CheegerD} is also isometric. Thus, as consequence of Proposition \ref{PQuotient} there is a unique 1-metric $\eta_\tau$ on $X_1\rightrightarrows X_0$ making of the projection $\pi=(X_1\times G_1\rightrightarrows X_0\times G_0,\eta\oplus \frac{1}{\tau}Q)\to (X_1\rightrightarrows X_0,\eta_\tau)$ a Riemannian groupoid submersion. 
	
	It is important to point out that by construction $\eta_\tau$ goes to $\eta$ as $\tau$ goes to $\infty$ and the $2$-action of $G_1\rightrightarrows G_0$ on $(X_1\rightrightarrows X_0,\eta_\tau)$ is still isometric when $\tau>0$. More importantly, the 1-parameter family of $1$-metrics $\eta_\tau$ on $X_1\rightrightarrows X_0$ varies smoothly with $\tau$ and extends smoothly to $\tau = 0$ with $\eta_0 = \eta$. Hence, $\eta_\tau$ with $\tau\geq  0$ is a deformation of $\eta$ by other $(G_1\rightrightarrows G_0)$-invariant metrics on  $X_1\rightrightarrows X_0$ which we call \emph{Cheeger groupoid deformation} of $\eta$. The reader is recommended to visit \cite[s. 6,1]{AB} for specific details about the last assertions in the classical case. 
\end{example}

We end this section by commenting that our notion of isometric $2$-action can be used to develop a $2$-equivariant analogous of the results proved in \cite{OV} which are aimed at extending classical Morse theory to the context of Lie groupoids and their differentiable stacks.

\begin{example}[2-Equivariant Morse theory on groupoids]\label{2-EquiMorse}
Let $\theta$ be an isometric $2$-action of a Lie $2$-group $G_1\rightrightarrows G_0$ on a proper Riemannian groupoid $(X_1\rightrightarrows X_0,\eta)$. Assume that $G_1$ is compact so that we can ensure that such a $2$-action always exists. Suppose that there is a $G_0$-invariant basic function $f:X_0\to \mathbb{R}$, where by basic we mean that $s_X^\ast f= t_X^\ast f$. Note that the critical points of $f$ come in families of $G_0$-invariant saturated submanifolds in $X_0$ so we may impose a Morse--Bott (normal) nondegenerate condition along those critical submanifolds. As it was proven in \cite{OV}, examples of these kinds of functions are provided by the components of a moment map which is associated to a Hamiltonian 2-action of a Lie $2$-torus on a $0$-symplectic groupoid in the sense of \cite{hsz}.

Hence, without so many changes along the proofs in \cite{OV} it is possible to prove the following results.
\begin{itemize}
\item A $2$-equivariant version of the Morse lemma around a $G_0$-invariant nondegenerate critical saturated submanifold of $f$. This can be done by following two possible approaches. The first one is by using the $2$-equivariant linearization from Proposition \ref{Lin1} together with the classical equivariant Morse Lemma proved in \cite{W}. The second one is by applying the ideas of the  proof of the classical Morse--Bott lemma provided in \cite[s. 4.2 \& Thm. 4.5]{Mein} and then by constructing an Euler--like multiplicative vector field which must be made $G$-invariant by averaging with respect to the $2$-action $\theta$. 
\item To describe the topological behavior of the level subgroupoids of $f$ in a $2$-equivariant way. Firstly, the $2$-action $\theta$ induces $2$-actions of $G_1\rightrightarrows G_0$ on the level subgroupoids of $f$ since this is $G_0$-invariant. Secondly, the gradient vector field of $f$ with respect to $\eta$ is a $G$-invariant multiplicative vector field so that its flow determines a 1-parametric family of $2$-equivariant (local) Lie groupoid automorphisms of $X_1\rightrightarrows X_0$. As expected, such a flow allows to study how are the topology changes of our Lie groupoid whether or not we cross by a critical subgroupoid level of $f$. Namely, if $f^{-1}[a,b]$ does not contain critical points of $f$ then the subgroupoid levels $X_1^a$ and $X_1^b$ are $2$-equivariant isomorphic. Otherwise, if $f^{-1}[a,b]$ contains no critical points besides a $G_0$-invariant nondegenerate critical saturated submanifold $S$ then the subgroupoid level $X_1^b$ is $2$-equivariant homotopy equivalent to $X_1^a\cup _{\partial D_-(G_S)}D_-(G_S)$. Here $D_-(G_S)$ and $\partial D_-(G_S)$ are the groupoids of disks and spheres of $S$ which are defined with respect to $\eta$ and also depend on $f$, see \cite{OV}. In our case, the $2$-action $\theta$ also induces canonical $2$-actions of $G_1\rightrightarrows G_0$ on those groupoids. It is important to mention that these results strongly depend on the fact that our groupoid metric $\eta$ is invariant by the $2$-action of $G_1\rightrightarrows G_0$.
\item To construct an equivariant Morse--Bott double cochain complex which computes the equivariant cohomology of a 2-action as defined in \cite{OBT}. Such a double complex is an extension of the equivariant Austin--Braam' complex introduced in \cite[s. 4]{AB} to the 2-equivariant setting. Our approach mainly depends on the fact that the gradient vector field of $f$ with respect to $\eta$ is a $G$-invariant multiplicative vector field which implies that the stable and unstable groupoids of any $G_0$-invariant nondegenerate critical saturated submanifold of $f$ inherit $2$-actions of $G_1\rightrightarrows G_0$ induced by $\theta$. The details needed to construct such an equivariant Morse--Bott double cochain complex were already completed in \cite[s. 9.3]{OV}.
\end{itemize}
\end{example}

\section{Isometries and geometric Killing vector fields}\label{S:4}

Let $G$ be a Lie group, with Lie algebra $\mathfrak{g}$, acting isometrically on a Riemannian manifold $(M,\eta^{(0)})$. It is well known that the fundamental vector field of each element in $\mathfrak{g}$ determines a Killing vector field on $M$ so that we get a Lie algebra homomorphism from $\mathfrak{g}$ to the finite dimensional Lie subalgebra $\mathfrak{o}(M,\eta^{(0)})\leq \mathfrak{X}(M)$ of Killing vector fields on $M$. The aim of this section is to bring an infinitesimal description of an isometric Lie $2$-action. Our approach will lead us to the study an algebra of transversal infinitesimal isometries associated to any Riemannian $n$-metric on a Lie groupoid. These transversal isometries will give rise to a notion of geometric Killing vector field on a quotient Riemannian stack.

We start by describing what would be our attempt to set the \emph{diffeomorphism group} of a Lie groupoid. Some of the references we shall be following throughout are \cite{Ma,OW} and \cite[App. D]{An}. A \emph{bisection} of a Lie groupoid $X_1\rightrightarrows X_0$ is a smooth map $\sigma:X_0\to X_1$ such that $s_X\circ \sigma=\textnormal{id}_{X_0}$ and $\iota_{\sigma}:X_0\to X_0$ defined by $\iota_{\sigma}(x):=t_X(\sigma(x))$ is a diffeomorphism. The set of all bisections of $X_1\rightrightarrows X_0$ will be denoted by $\textnormal{Bis}(X)$. This has the structure of an infinite-dimensional Lie group where the multiplication of two bisections $\sigma$ and $\sigma'$ is given by $\sigma\bullet \sigma'(x):=\sigma((t_X\circ \sigma')(x))*\sigma'(x)$ for all $x\in X_0$; see for instance \cite{SW} and \cite[s. 1.4]{Ma}. Let us denote by $\textnormal{Aut}(X)$ the group of Lie groupoid automorphisms of $X_1\rightrightarrows X_0$. Given a bisection $\sigma \in \textnormal{Bis}(X)$ one has an inner automorphism $I_{\sigma}:X_1\to X_1$ defined by

\begin{equation}\label{BisAut}
I_{\sigma}(p):=\sigma(t_X(p))*p* i_X(\sigma(s_X(p))).
\end{equation}

Clearly, $I_{\sigma}$ covers the map $\iota_{\sigma}$. This inner automorphism allows us to define what we call the \emph{crossed module of automorphisms of a Lie groupoid} $(\textnormal{Aut}(X),\textnormal{Bis}(X),I,\alpha)$ where the map $\alpha$ is defined as $\alpha_{\Phi}(\sigma):=\Phi\circ \sigma \circ \phi^{-1}$ for all $\Phi \in \textnormal{Aut}(X)$ covering $\phi:X_0\to X_0$ and $\sigma \in \textnormal{Bis}(X)$. Accordingly, we have a 2-group $\textnormal{Bis}(X)\ltimes \textnormal{Aut}(X)\rightrightarrows \textnormal{Aut}(X)$ called the \emph{2-group of Lie groupoid automorphisms}. The first property we obtain of such a $2$-group is provided below.
\begin{proposition}
	The orbit space of the 2-group of Lie groupoid automorphisms equals the set of Lie groupoid automorphisms up to smooth natural equivalences. Namely:
$$\textnormal{Aut}(X)/\textnormal{Bis}(X)=\left\lbrace [\Phi]\,|\, \Phi\in \textnormal{Aut}(X),\quad \Psi\sim \Phi \Leftrightarrow \exists_{\alpha}\left(\Psi\stackrel{\alpha}{\Rightarrow}\Phi\right)\right\rbrace.$$
\end{proposition}
\begin{proof}
	Let us describe the orbit of an element in $\Phi\in \textnormal{Aut}(X)$ covering $\phi$. If we pick $\sigma\in \textnormal{Bis}(X)$ and $\Psi \in \textnormal{Aut}(X)$ covering $\psi$ such that $I_{\sigma}\Phi=\Psi$ then it follows that 
$$\Psi(p)=I_{\sigma}(\Phi(p))=\sigma(\phi(t_X(p)))*\Phi(p)*i_X(\sigma(\phi(s_X(p)))),$$
for some $p \in X_1$ so that $	\Psi(p)*\sigma(\phi(s_X(p)))=\sigma(\phi(t_X(p)))*\Phi(p)$. Therefore, by setting $\alpha:=\sigma\circ \phi$ we get a smooth natural transformation $\Phi\stackrel{\alpha}{\Rightarrow}\Psi$. Conversely, note that if $\Phi\stackrel{\alpha}{\Rightarrow}\Psi$  a smooth natural transformation then for some $p\in X$ it holds $\alpha(t_X(p))*\Phi(p)=\Psi(p)*\alpha(s_X(p))$, thus obtaining that $t_X(\Phi(p))=s_X(\alpha(t_X(p)))$ and $s_X(\Psi(p))=t_X(\alpha(s_X(p)))$. On the one hand, by setting $\sigma:=\alpha\circ \phi^{-1}$ it follows that $s_X\circ\sigma=\textnormal{id}_{X_0}$. On the other hand, observe that
$$ \psi(s_X(p))=t_X(\alpha(s_X(p)))=t_X(\sigma(\phi( s_X)))(p)=\iota_{\sigma}(\phi(s_X(p))).$$
Hence, we have obtained that $i_{\sigma}=\psi\circ \phi^{-1}$ so that it is a diffeomorphism. 
\end{proof}
Some simple examples which illustrate the naturality of the previous result are the following.

\begin{example}
If $G\rightrightarrows *$ is a Lie group then $\textnormal{Bis}(G)\simeq G$. We may think of $\textnormal{Bis}(G)$ as the subset in $\textnormal{Aut}(G)$ determined by conjugations with respect to the elements in $G$. This implies that the crossed module of automorphisms of $G$ is $(\textnormal{Aut}(G),G,j,c)$ where $j$ is the inclusion and $c$ is the identity representation. Thus, the orbit space $\textnormal{Aut}(G)/G$ corresponds to the set of automorphisms of $G$ up to conjugations. That is, the outer automorphisms of $G$.
\end{example}

\begin{example}\label{ExampleOrbit2-group}
Let $\pi:M\to N$ be a surjective submersion and let $M\times_{N}M\rightrightarrows M$ denote its corresponding submersion groupoid. A straightforward computation shows that $\textnormal{Aut}(M\times_{N}M)$ is in one-to-one correspondence with $\textnormal{Aut}(\pi)$ which stands for the set of pairs $(\widetilde{f},f)\in \textnormal{Diff}(M)\times \textnormal{Diff}(N)$ commuting with $\pi$ and that $\textnormal{Bis}(M\times_{N}M)$ corresponds to $\textnormal{Gau}(\pi)$ that is the set of pairs $(\widetilde{f},\textnormal{id})\in \textnormal{Aut}(\pi)$. In this case the crossed module of automorphisms of $M\times_{N}M$ is $(\textnormal{Aut}(\pi),\textnormal{Gau}(\pi),j,c)$ where $j$ is the inclusion and $c$ is the representation by conjugations. Therefore, the orbit space
$$\textnormal{Aut}(M\times_{N}M)/\textnormal{Bis}(M\times_{N}M)\simeq \textnormal{Aut}(\pi)/\textnormal{Gau}(\pi)\simeq  \textnormal{Diff}(N).$$
\end{example}

\begin{comment}
A more interesting example appears when analyzing the gauge groupoid of a principal bundle. Namely:
\begin{example}
Suppose that $\pi:P\to M$ is a principal $G$-bundle such that $M$ is connected and let us consider its corresponding gauge groupoid $P\times_{G}P\rightrightarrows P/G$. On the one hand, a simple computation shows that $\textnormal{Aut}(P\times_{G}P)$ may be identified with the set of diffeomorphisms on $P$ that are equivariant with respect to a group homomorphism $G\to G$. On the other hand, $\textnormal{Bis}(P\times_{G}P)$ is identified with the set of pairs $(f,g)\in \textnormal{Diff}(P)\times \textnormal{Diff}(P)$ that are also equivariant with respect to a group homomorphism $G\to G$ and satisfy that $\pi\circ f=\overline{f}\circ \pi$ and $\pi\circ g=\pi$. Here $\overline{f}\in \textnormal{Diff}(P/G)$ is completely determined by $f$. Let us denote by $I:\textnormal{Bis}(P\times_{G}P)\to \textnormal{Aut}(P\times_{G}P)$ the group homomorphism that sends a bisection $\sigma$ to the automorphisms $I_\sigma$ as defined in Equation \eqref{BisAut}. Hence, under the identifications mentioned above it is simple to check that $I$ is surjective and that $\textnormal{ker}(I)$ may be identified with the $G$-equivariant diffeomorphisms $g$ on $P$ such that $\pi\circ g=\pi$. In consequence, the orbit space $\textnormal{Aut}(P\times_{G}P)/\textnormal{Bis}(P\times_{G}P)$ is trivial.
\end{example}
\end{comment}

Let $G_1\rightrightarrows G_0$ be a Lie $2$-group acting on $X_1\rightrightarrows X_0$ by the left.

\begin{lemma}\label{2ActionDecomposition}
The normal subgroup $H=\textnormal{ker}(s_{G})$ acts on $X_1\rightrightarrows X_0$ by bisections and $G_0$ acts by Lie groupoid automorphisms. Moreover, the right multiplication map defined on $s_X$-fibers for each arrow is $H$-equivariant. 
\end{lemma}
\begin{proof}
	Consider $h\in H$ and define the map $\sigma_{h}:X_0\to X_1$ as $\sigma_h(x):=h1_x.$ It is clear that $s_X(\sigma_h(x))=x$ and $t(\sigma_h(x))=\rho(h)x$ so that $\sigma_h$ is a well defined bisection.  Let us now take $g\in G_0$ and define the map $\Sigma_g:X_1\to X_1$ as $\Sigma_g(x):=1_{g}x$. Equation \eqref{MultAction} implies that
	\[1_{g}(x*y)=(1_g*1_g)(p*q)=(1_{g}p)*(1_{g}q),\]
	for all $g\in G_0$ and $(p,q)\in X_2$, thus obtaining that $\Sigma_g$ is a Lie groupoid morphism which clearly satisfies $(\Sigma_g)^{-1}=\Sigma_{g^{-1}}$. Note that $s_{X}(hp)=s_{X}(p)$ for all $p\in X_1$ and $h\in H$ so that the left action of $H$ on $X_1$ preserves the $s_X$-fibers. Therefore, for each $y\xleftarrow[]{\it p}x$ the right action $R_{p}:s_X^{-1}(y)\to s_X^{-1}(x)$ satisfies that
	\[R_{p}(hq)=(hq)*(1_{e}p)=(h*1_e)(q*p)=hR_{p}(q).\]
	In consequence, $R_p$ is $H$-equivariant as claimed.
\end{proof}
It is clear that the same result can be obtained if we consider right 2-actions instead of left ones. Let $(G,H,\rho,\alpha)$ denote the crossed module of Lie groups associated to $G_1\rightrightarrows G_0$. Then:
\begin{lemma}\label{LemmaIso1}
There is a natural morphism of crossed modules of Lie groups $(\sigma,\Sigma):(G,H,\rho,\alpha) \to (\textnormal{Aut}(X),\textnormal{Bis}(X),I,\alpha)$ where $\sigma_h$ and $\Sigma_g$ are defined as in Lemma \ref{2ActionDecomposition}.
\end{lemma}
\begin{proof}
Let us check that $\Sigma \circ \rho=I\circ \sigma$ and $\sigma_{\alpha_{g}(h)}=\alpha_{\Sigma_g}(\sigma_h)$ for all $g \in G$ and $h\in H$. Firstly, for $ p\in X_1$ and $h\in H$ we obtain
\begin{eqnarray*}
		I_{\sigma_h}(p)&=&\sigma_h(t_X(p))*p*i(\sigma_h(s_X(p)))=(h1_{t_X(p)})*p*i(h1_{s_X(p)})\\
		&=&(h*e)(1_{t_X(p)}*p)*i(h1_{s_X(p)})=hp*i_{G}(h)1_{s(p)}=1_{\rho(h)}p=\Sigma_{\rho(h)}(p).
	\end{eqnarray*}
	Secondly, for $x \in X_0$, $g\in G_0$ and $h\in H$ we get
$$
\alpha_{\Sigma_g}(\sigma_h)(x) =\Sigma_g(h1_{(g^{-1}x)})=1_gh1_{g^{-1}}1_x=\alpha_{g}(h)1_x=\sigma_{\alpha_g(h)}(x).
$$
\end{proof}
We are now in conditions to define an \emph{infinitesimal 2-action} associated to a right Lie 2-group action. In order to do so we need to introduce the following structure. Let $A_X$ be the Lie algebroid of $X_1\rightrightarrows X_0$ and for each $\xi \in \mathfrak{g}_0$ we denote its fundamental vector fields as $\tilde{\xi}$. We also denote by $\mathfrak{X}_{m}(X)$ the set of multiplicative vector fields on $X_1\rightrightarrows X_0$ and by $(\mathfrak{X}_{m}(X),\Gamma(A_{X}),\delta,D)$ its associated crossed module. Here we have that $\delta(\alpha)=\alpha^r-\alpha^l$ and $D_{(\xi,v)}\alpha=[\xi,\alpha^r]|_{X_0}$ for all $\alpha\in \Gamma(A_X)$ and $(\xi,v)\in \mathfrak{X}_{m}(X)$. The reader is recommended to visit \cite[s. 7.1]{OW} for getting specific details.

\begin{theorem}\label{InfinitesimalAction}
Suppose that we have a right Lie 2-action of $G_1\rightrightarrows G_0$ on $X_1\rightrightarrows X_0$. Let $(\mathfrak{g},\mathfrak{h},\partial,\mathcal{L})$ be the crossed module of Lie algebras associated to the Lie $2$-algebra of $G_1\rightrightarrows G_0$. Then there is a canonical homomorphism of Lie 2-algebra $j=(j_{-1},j_0):(\mathfrak{g},\mathfrak{h},\partial,\mathcal{L})\to (\mathfrak{X}_{m}(X),\Gamma(A_{X}),\delta,D)$ defined by $j_{-1}(\xi)=\left.\tilde{\xi}\right|_{X_0}$ and  $j_0(\zeta)=(\tilde{1_{\zeta}},\tilde{\zeta})$ for all $\xi \in \mathfrak{h}$ and $\zeta \in \mathfrak{g}$.
\end{theorem}
\begin{proof}
To see that $j$ is a well defined morphism we have to check that for any $\xi \in \mathfrak{h}$ its fundamental vector field $\tilde{\xi}$ belongs to $\mathfrak{X}_{inv}^{s}(X)$ and that for any $\zeta \in \mathfrak{g}$ it holds that $(\tilde{1_{\zeta}},\tilde{\zeta})$ is in $\mathfrak{X}_{m}(X)$. The latter assertion is clear since if $\exp{t\zeta}\in G_0$ then Lemma \ref{2ActionDecomposition} implies that the pair of flows $(\varphi_t^{\tilde{1_\zeta}},\varphi_t^{\tilde{\zeta}})$ determines a Lie groupoid morphism. Now, if $\xi \in \mathfrak{h}$ then again from Lemma \ref{2ActionDecomposition} it follows that the flow of its fundamental vector field lies inside the $s$-fibers since $s_{X}(\varphi_t^{\tilde{\xi}}(p))=s_{X}(p\exp(t\xi))=s_{X}(p)$. Therefore, $\tilde{\xi}\in \mathfrak{X}^s(X)$. Furthermore, for $y\xleftarrow[]{\it p}x$ and $q\in s_X^{-1}(y)$ one has that
$$
R_{p}(\varphi^{\tilde{\xi}}_t(q)) =(q\exp(t\xi))*p=(q*p)(\exp{t\xi}*1_e)=\varphi_t^{\tilde{\xi}}(q*p)=\varphi_t^{\tilde{\xi}}(R_p(q)),
$$
thus obtaining that $d(R_p)_q(\tilde{\xi}_q)=\tilde{\xi}_{q*p}$ so that $\tilde{\xi}\in \mathfrak{X}_{inv}^s(X)$ and $\left.\tilde{\xi}\right|_{X_0}\in \Gamma(A_X)$.

Let us finally verify that for $\xi\in \mathfrak{h}$ and $\zeta \in \mathfrak{g}$ it satisfies that $\delta(j_{-1}(\xi))=j_{0}(\partial \xi)$ and $j_{-1}(\mathcal{L}_{\zeta}\xi)=D_{j_0(\zeta)}(j_{-1}(\xi))$. On the one hand, by using the flow of vector field $\delta(j_0(\xi))=\left.\tilde{\xi}\right|_{X_0}^r-\left.\tilde{\xi}\right|_{X_0}^l$ and Equation \eqref{MultAction} we get
\begin{eqnarray*}
\varphi_t^{\delta(j_{-1}(\xi))}(p)&=&\varphi_{t}^{\tilde{\xi}}(1_{t_{X}(p)})*p*\iota_X(\varphi_t^{\tilde{\xi}}(1_{s_X(p)}))=1_{t_{X}(p)}\exp(t\xi)*p*i_X(1_{s_{X}(p)}\exp(t\xi))\\
&=&(1_{t_X(p)}*p)\exp(t\xi)*1_{s_X(p)}i_G(\exp(t\xi))=p(\exp(t\xi)*i_G(\exp(t\xi)))\\
&=&p(1_{t_G(\exp(t\xi))})=p1_{\exp(t\partial(\xi))}=p\exp(t1_{\partial \xi})=\varphi_t^{{j_0(\partial \xi)}}(p).
	\end{eqnarray*}
	Hence, $\delta(j_{-1}(\xi))=j_{0}(\partial \xi)$. On the other hand, observe that
	\[    D_{j_0(\zeta)}(j_{-1}(\xi))=\left.\left[\tilde{1_{\zeta}},\tilde{\xi}\right]\right|_{X_0}=\left.\tilde{\left[1_{\zeta},\xi\right]}\right|_{X_0}=(\left.\tilde{\mathcal{L}_{\zeta}(\xi))}\right|_{X_0}=j_{-1}(\mathcal{L}_{\zeta}\xi).\]
\end{proof}

Let us now put a Riemannian structure into play. Suppose that $X_1\rightrightarrows X_0$ can be equipped with a 0-metric $\eta$ and consider the following sets 
$$\textnormal{Bis}_{\eta}(X)=\left\lbrace \sigma \in \textnormal{Bis}(X)\,|\, \iota_{\sigma}^*\eta=\eta\right\rbrace\quad\textnormal{and}\quad \mathrm{Iso}(X,\eta)=\left \lbrace (\Phi,\phi)\in \mathrm{Aut}(X)\,|\, \phi^*\eta=\eta \right\rbrace.$$

\begin{proposition}\label{IsoStrong1}
The quadruple $(\mathrm{Iso}(X,\eta),\textnormal{Bis}_{\eta}(X),I,\alpha)$ determines a sub-crossed module structure  of $(\textnormal{Aut}(X),\textnormal{Bis}(X),I,\alpha)$.
\end{proposition}
\begin{proof}
It is clear that $I(\textnormal{Bis}_{\eta}(X))\subseteq \mathrm{Iso}(X,\eta)$. As $I_{\alpha_{\Phi}(\sigma)}=\Phi I_{\sigma}\Phi^{-1}$ for all $\sigma \in \textnormal{Bis}_{\eta}(X)$ and $\Phi \in \textnormal{Iso}(X,\eta)$ then when restricting to unities we have that $\iota_{\alpha_{\Phi}(\sigma)}=\phi\circ \iota_{\sigma}\circ \phi^{-1}$, thus obtaining an isometry.
\end{proof}
The Lie 2-group associated to the crossed module $(\mathrm{Iso}(X,\eta),\textnormal{Bis}_{\eta}(X),I,\alpha)$ will be called \emph{Lie 2-group of strong isometries} of $(X_1\rightrightarrows X_0,\eta)$. 
\begin{example}
Let $(M,\eta)$ be an $n$-dimensional Riemannian manifold and let $\pi:O(M)\to M$ denote the corresponding $O(n,\mathbb{R})$-principal bundle of orthonormal frames. After fixing an $\textnormal{Ad}$-invariant inner product $\langle\cdot,\cdot \rangle$ on the Lie algebra $\mathfrak{o}(n,\mathbb{R})$ and taking the connection 1-form $\omega\in \Omega^1(O(M),\mathfrak{o}(n,\mathbb{R}))$ associated to the Levi--Civita connection on $(M,\eta)$ we can define a Riemannian metric on $O(M)$ as
$$\tilde{\eta}(X,Y)=\eta(d\pi(X),d\pi(Y))+ \langle\omega(X),\omega(Y) \rangle.$$

It follows that $\pi:O(M)\to M$ becomes a Riemannian submersion and its corresponding submersion groupoid $O(M)\times_{M}O(M)\rightrightarrows O(M)$ inherits an induced $0$-metric which we also denote by $\eta$ \cite{dHF,O'n}. Therefore,
$$\textnormal{Bis}_{\eta}(O(M)\times_{M}O(M))\simeq \textnormal{Gau}(O(M),\theta)\quad\textnormal{and}\quad  \textnormal{Iso}(O(M)\times_{M}O(M),\eta)\simeq \textnormal{Aut}(O(M),\theta),$$
where $\theta\in \Omega^{1}(O(M),\mathbb{R}^n)$ is the canonical $1$-form, $\textnormal{Aut}(O(M),\theta)$ is the group of bundle isomorphisms preserving $\theta$ and
$\textnormal{Gau}(O(M),\theta)$ is its normal subgroup of bundle isomorphisms covering the identity (compare Example \ref{ExampleOrbit2-group}). Hence, from \cite[p. 236]{KN} we get that the orbit space
$$\textnormal{Iso}(O(M)\times_{M}O(M),\eta)/\textnormal{Bis}_{\eta}(O(M)\times_{M}O(M))\simeq  \textnormal{Iso}(M,\eta).$$
\end{example}

Let us now consider the sets 
$$\Gamma_{\eta}(A_X)=\left\lbrace \alpha \in \Gamma(A_X)\,|\, \rho(\alpha) \in \mathfrak{o}(X_0,\eta)\right\rbrace\quad\textnormal{and}\quad \mathfrak{o}_{m}(X)=\left\lbrace (\xi,v)\in  \mathfrak{X}_{m}(X)\,|\, v \in \mathfrak{o}(X_0,\eta)\right\rbrace$$
where $\mathfrak{o}(X_0,\eta)$ denotes the Lie algebra of Killing vector fields of $(X_0,\eta)$. In these terms we may describe the infinitesimal version of the Lie $2$-group of strong isometries as follows. 
\begin{proposition}\label{Strongkilling}
The quadruple $(\mathfrak{o}_{m}(X),\Gamma_{\eta}(A_X),\delta,D)$ defines a sub-crossed module structure of $(\mathfrak{X}_{m}(X),\Gamma(A_{X}),\delta,D)$.
\end{proposition}
\begin{proof}
Firstly, note that $\mathfrak{o}_{m}(X)$ is a Lie subalgebra of $\mathfrak{X}_{m}(X)$ since $\mathfrak{o}(X_0,\eta)$ is a Lie algebra. Secondly, if $\alpha, \beta \in \Gamma_{\eta}(A_X)$ then $\rho([\alpha,\beta])=[\rho(\alpha),\rho(\beta)]\in \mathfrak{o}(X_0,\eta)$ so that $[\alpha,\beta]\in \Gamma_{\eta}(A)$. If $(\xi,v)\in \mathfrak{o}_{m}(X)$ and $\alpha \in \Gamma_{\eta}(A)$ then we have by definition that $D_{\xi}(\alpha)=\left.[\xi,\alpha^r]\right|_{X_0} \in \Gamma(A_X)$. However, the equivariance identity implies that $\delta(D_{\xi}(\alpha))=[\xi,\delta(\alpha)]$. Therefore, it holds that $\left.\delta(D_{\xi}\alpha)\right|_{X_0}=\left.[\xi,\delta(\alpha)]\right|_{X_0}$ which is the same thing that saying
$\rho(D_{\xi}\alpha)=[v,\rho(\alpha)] \in \mathfrak{o}(X_0,\eta)$ since $v$ is also a Killing vector field.
\end{proof}
The Lie $2$-algebra associated to the crossed module $(\mathfrak{o}_{m}(X),\Gamma_{\eta}(A_X),\delta,D)$ will be called \emph{Lie 2-algebra of strong multiplicative Killing vector fields} of $(X_1\rightrightarrows X_0,\eta)$.

Summing up, as consequence of Lemma \ref{Rmk1} we can now provide an infinitesimal description of an isometric Lie $2$-group action. Indeed:

\begin{corollary}\label{InfinitesimalIsometriesRep}
Let $\theta$ be an isometric right 2-action of a Lie $2$-group $G_1\rightrightarrows G_0$ on a Riemannian groupoid $(X_1\rightrightarrows X_0,\eta)$. Denote by $(G,H,\rho,\alpha)$ to the crossed module of Lie groups associated to $G_1\rightrightarrows G_0$ and by $(\mathfrak{g},\mathfrak{h},\partial,\mathcal{L})$ to its corresponding crossed module of Lie algebras. Then:
\begin{itemize}
\item there is a morphism of crossed modules of Lie groups $$(\sigma,\Sigma):(G,H,\rho,\alpha) \to (\textnormal{Iso}(X,\eta),\textnormal{Bis}_\eta(X),I,\alpha),$$that is defined as in Lemma \ref{LemmaIso1}, and
\item there is a morphism of crossed modules of Lie algebras $$(j_{-1},j_0):(\mathfrak{g},\mathfrak{h},\partial,\mathcal{L})\to (\mathfrak{o}_{m}(X),\Gamma_{\eta}(A_X),\delta,D),$$
which is defined in Theorem \ref{InfinitesimalAction}.
\end{itemize}
\end{corollary}

\subsection{Transversal isometries}\label{S:4:1}

 Recall that a $0$-metric on a Lie groupoid $X_1\rightrightarrows X_0$ is a Riemannian metric $\eta$ on $X_0$ which is transversely invariant by the canonical left action of $X_1\rightrightarrows X_0$ on $X_0$, compare \cite{dHF,PPT}. Note that this is the same that requiting that $\eta$ is transversely invariant by the action of the group of bisections $\textnormal{Bis}(X)\times X_0\to X_0$ which is defined by $\sigma\cdot x:=\iota_{\sigma}(x)$. It is clear that this action preserves the orbits so that it induces a well defined action on the normal space of an orbit. In consequence, $\eta$ is a 0-metric if and only if for all $\sigma \in \textnormal{Bis}(X)$ the map $\overline{d\iota_{\sigma}}:(\nu_x(\mathcal{O}),\overline{\eta})\to (\nu_{\iota_{\sigma}(x)}(\mathcal{O}),\overline{\eta})$ is a linear isometry since we may identify $\nu(\mathcal{O})\cong T\mathcal{O}^\perp$. Therefore, motivated by this fact and what we did in the previous section we now plan to weaken the condition for a diffeomorphism to be an isometry by imposing instead a transversal isometric condition along groupoid orbits. This will lead us to define a Lie $2$-algebra of transverse infinitesimal isometries with respect to any Riemannian groupoid $n$-metric, which at the end turns out to be Morita invariant.

Let $\Phi:X_1\to X_1$ be a Lie groupoid automorphism covering $\phi:X_0\to X_0$. From now on we assume the identification $\nu(\mathcal{O})\cong T\mathcal{O}^\perp$ for each groupoid orbit $\mathcal{O}$ in $(X_0,\eta^{(0)})$ without stating it explicitly unless it is necessary. The diffeomorphism $\phi:X_0\to X_0$ is said to be a \emph{transversal isometry} of $(X_0,\eta)$ if $\overline{d\phi}:\nu(\mathcal{O}_x)\to \nu(\mathcal{O}_{\phi(x)})$ is a fiberwise isometry for every groupoid orbit $\mathcal{O}_x$ in $X_0$. Consider the set
$$\textnormal{Iso}_{\textnormal{w}}(X,\eta)=\left \lbrace (\Phi,\phi)\in \mathrm{Aut}(X)\,|\, \phi\ \textnormal{transversal isometry of}\ (X_0,\eta)\right\rbrace.$$
Note that for every $\sigma\in \textnormal{Bis}(X)$ it follows that $\iota_\sigma$ is a transversal isometry of $(X_0,\eta)$. Thus, by arguing as in Proposition \ref{IsoStrong1} we easily get that:
\begin{lemma}
The quadruple $(\mathrm{Iso}_{\textnormal{w}}(X,\eta),\textnormal{Bis}(X),I,\alpha)$ determines a sub-crossed module structure  of $(\textnormal{Aut}(X),\textnormal{Bis}(X),I,\alpha)$.
\end{lemma}
The Lie $2$-group determined by the crossed module $(\mathrm{Iso}_{\textnormal{w}}(X,\eta),\textnormal{Bis}(X),I,\alpha)$ is called \emph{Lie 2-group of weak isometries} of $(X_1\rightrightarrows X_0,\eta)$. Let us denote by $\mathfrak{o}^{\textnormal{w}}(X_0,\eta)$ the set of vector fields on $X_0$ which are covered by multiplicative vector fields on $X_1$ whose flow determines a (local) transversal isometry of $(X_0,\eta)$ and consider the set $\mathfrak{o}_{m}^{\textnormal{w}}(X)=\left\lbrace (\xi,v)\in  \mathfrak{X}_{m}(X)\,|\, v \in \mathfrak{o}^{\textnormal{w}}(X_0,\eta)\right\rbrace$. Observe that if $v_1,v_2\in \mathfrak{o}^{\textnormal{w}}(X_0,\eta)$ then for $t$ small enough we have that the commutator flow $\varphi_{-\sqrt{t}}^{v_2}\varphi_{-\sqrt{t}}^{v_1}\varphi_{\sqrt{t}}^{v_2}\varphi_{\sqrt{t}}^{v_1}$ is a transversal isometry so that $\mathfrak{o}^{\textnormal{w}}(X_0,\eta)$ is a Lie subalgebra of $\mathfrak{X}(X_0)$. From \cite[Thm. D]{SW} we know that $\textnormal{Lie}(\textnormal{Bis}(X))$ is identified with $\Gamma(A_X)$. Therefore, by using similar arguments as those in Proposition \ref{Strongkilling} we obtain a description of the Lie $2$-group of weak isometries of $(X_1\rightrightarrows X_0,\eta)$. Namely:
\begin{proposition}
The quadruple 
$(\mathfrak{o}_{m}^{\textnormal{w}}(X),\Gamma(A_X),\delta,D)$ defines a sub-crossed module structure of $(\mathfrak{X}_{m}(X),\Gamma(A_{X}),\delta,D)$.
\end{proposition}
Accordingly, the Lie $2$-algebra associated to the crossed module $(\mathfrak{o}_{m}^{\textnormal{w}}(X),\Gamma(A_X),\delta,D)$ will be called \emph{Lie 2-algebra of weak multiplicative Killing vector fields} of $(X_1\rightrightarrows X_0,\eta)$.

\begin{remark}\label{WeakKillingNerve}
	Let us suppose that we are equipped with an $n$-metric $\eta^{(n)}$ on $X_n$ and that $(\xi,v)$ is a multiplicative vector field on $X_1\rightrightarrows X_0$ with $v \in \mathfrak{o}^{\textnormal{w}}(X_0,\eta^{(0)})$. It is simple to see that the fact that $s_X:X_1\to X_0$ (or $t_X:X_1\to X_0$) is a Riemannian submersion clearly implies that $\xi \in \mathfrak{o}^{\textnormal{w}}(X_1,\eta^{(1)})$. More importantly, if $\xi_n$ denotes the vector field on $X_n$ induced by $(\xi,v)$ for all $n\geq 2$ then if follows that $\xi_n \in \mathfrak{o}^{\textnormal{w}}(X_n,\eta^{(n)})$ since the face maps $X_n\to X_{n-1}$ are Riemannian submersions. As a consequence of this we have that the notion of weak multiplicative Killing vector field can be extended to a notion associated to any Riemannian $n$-metric. This immediately implies that the Lie 2-algebra of weak multiplicative Killing vector fields $(\mathfrak{o}_{m}^{\textnormal{w}}(X),\Gamma(A_X),\delta,D)$ can be thought of as an algebraic object associated to any Riemannian $n$-metric on $X_n$.
\end{remark}

\subsection{Morita invariance of transversal Killing vector fields}\label{S:4:2}

In this subsection we apply some of the results from \cite{OW} to our context in order to define a notion of geometric Killing vector field on a quotient Riemannian stack. Then we use some of the results from \cite{CMS} to describe several interesting features that these kinds of vector fields have.

Let us start by introducing some necessary terminology.  A \emph{Morita map} is a groupoid morphism $\phi:(X_1\rightrightarrows X_0) \to (Y_1\rightrightarrows Y_0)$ which is \emph{fully faithful} and \emph{essentially surjective}, in the sense that the source/target maps define a fibred product of manifolds $X_1 \cong (X_0 \times X_0) \times_{ (Y_0\times Y_0)} Y_1$ and that the map $Y_1 \times _{Y_0}X_0\to X_0$ sending $(\phi^0(x)\to y)\mapsto y$ is a surjective submersion \cite{dH,MM}. An important fact shown in \cite[Thm. 4.3.1]{dH} is that a Lie groupoid morphism is a Morita map if and only if it yields an isomorphism between \emph{transversal data}. That is, the morphism must induce: a homeomorphism between the orbit spaces, a Lie group isomorphism $X_x\cong Y_{\phi^0(x)}$ between the isotropies and isomorphisms between the normal representations $X_x\curvearrowright \nu_x\to Y_{\phi^0(x)}\curvearrowright \nu'_{\phi^0(x)}$. We think of a quotient stack as a Lie groupoid up to Morita equivalence in the sense that two Lie groupoids $X$ and $Y$ define the same stack if there is a third groupoid $Z$ and Morita maps $Z\to X$ and $Z\to Y$ \cite{dH}. These two Morita maps may be assumed to be Morita fibrations (surjective submersion at the level of objects) \cite[p. 131]{MM}. The quotient stack associated to the Lie groupoid $X_1\rightrightarrows X_0$ will be denoted by $[X_0/X_1]$. 

Two Riemannian metrics $\eta_1$ and $\eta_2$ on $X_1\rightrightarrows X_0$ are said to be \emph{equivalent} if they induce the same inner product on the normal vector spaces over the groupoid orbits \cite{dHF2}. More generally, we define a \emph{Riemannian Morita map} (\emph{fibration}) $\phi: (Z_1 \rightrightarrows Z_0) \to (X_1\rightrightarrows X_0)$ as a Morita map between Riemannian groupoids that induces isometries on the normal vector spaces to the groupoid orbits $\nu_z(\mathcal{O}^Z)\to \nu_{\phi(z)}(\mathcal{O}^X)$ (Riemannian submersion at the level of objects). By using this terminology we have that  $\eta_1$ and $\eta_2$ are equivalent if and only if the identity $\textnormal{id}: (X_1 \rightrightarrows X_0,\eta_1) \to (X_1\rightrightarrows X_0,\eta_2)$ is a Riemannian Morita map.

Following \cite[s. 6]{OW}, given a Morita fibration $\phi: (Z_1 \rightrightarrows Z_0) \to (X_1\rightrightarrows X_0)$ we denote the set of projectable sections by
$$\Gamma(A_Z)^\phi=\lbrace \alpha \in \Gamma(A_Z):\textnormal{there exists}\ \alpha'\in \Gamma(A_X)\ \textnormal{such that}\ \phi_\ast\alpha=\alpha'\phi\rbrace.$$

If $\alpha \in \Gamma(A_Z)$ then the surjectivity of $\phi$ at the level of objects implies that there exists at most one section $\alpha'\in \Gamma(A_X)$ such that  $\phi_\ast\alpha=\alpha'\phi$, so that it follows that there is a natural linear map $\phi_\ast : \Gamma(A_Z)^\phi\to \Gamma(A_X)$. We denote by $\Gamma(A_Z)^\phi \hookrightarrow  \Gamma(A_Z)$ the inclusion map. It is clear that we can similarly define the set of projectable multiplicative sections $\mathfrak{X}_m(Z)^\phi$, a natural map $\phi_\ast: \mathfrak{X}_m(Z)^\phi\to \mathfrak{X}_m(X)$ and an inclusion $\mathfrak{X}_m(Z)^\phi\hookrightarrow \mathfrak{X}_m(Z)$. From \cite[Prop. 7.4]{OW} it follows that $(\mathfrak{X}_m(Z)^\phi,\Gamma(A_{Z})^\phi,\delta,D)$ is a sub-crossed module of $(\mathfrak{X}_m(Z),\Gamma(A_{Z}),\delta,D)$ and both maps $\phi_\ast: (\mathfrak{X}_m(Z)^\phi,\Gamma(A_{Z})^\phi,\delta,D)\to (\mathfrak{X}_m(X),\Gamma(A_{X}),\delta,D)$ and $(\mathfrak{X}_m(Z)^\phi,\Gamma(A_{Z})^\phi,\delta,D)\hookrightarrow(\mathfrak{X}_m(Z),\Gamma(A_{Z}),\delta,D)$  are morphisms of crossed-modules. More importantly,
$$(\mathfrak{X}_m(Z),\Gamma(A_{Z}),\delta,D) \hookleftarrow (\mathfrak{X}_m(Z)^\phi,\Gamma(A_{Z})^\phi,\delta,D) \xrightarrow[]{\it \phi_\ast} (\mathfrak{X}_m(X),\Gamma(A_{X}),\delta,D),$$
are quasi-isomorphisms of crossed modules. For specific details visit \cite{OW}.
\begin{lemma}\label{MoritaLemma1}
Let $\phi: (Z_1 \rightrightarrows Z_0,\eta^Z) \to (X_1\rightrightarrows X_0,\eta^X)$ be a Morita Riemannian fibration. Then
\begin{itemize}
\item $(\mathfrak{o}_m(Z)^\phi,\Gamma_{\eta}(A_{Z})^\phi,\delta,D)$ is a sub-crossed module of $(\mathfrak{o}_m(Z),\Gamma_\eta(A_{Z}),\delta,D)$,
\item the inclusion $(\mathfrak{o}_m(Z)^\phi,\Gamma_{\eta}(A_{Z})^\phi,\delta,D)\hookrightarrow(\mathfrak{o}_m(Z),\Gamma_{\eta}(A_{Z}),\delta,D)$ is a morphism of crossed modules, and
\item the projection $\phi_\ast: (\mathfrak{o}_m(Z)^\phi,\Gamma_\eta(A_{Z})^\phi,\delta,D)\to (\mathfrak{o}_m(X),\Gamma_\eta(A_{X}),\delta,D)$ is a morphism of crossed modules.
\end{itemize}
Moreover,
$$(\mathfrak{o}_m(Z),\Gamma_\eta(A_{Z}),\delta,D) \hookleftarrow (\mathfrak{o}_m(Z)^\phi,\Gamma_\eta(A_{Z})^\phi,\delta,D) \xrightarrow[]{\it \phi_\ast} (\mathfrak{0}_m(X),\Gamma_\eta(A_{X}),\delta,D),$$
are quasi-isomorphisms of crossed modules. Same conclusion holds true for the weak counterpart.
\end{lemma}
\begin{proof}
By using similar arguments as those in Lemma \ref{Rmk1} if follows that if $v$ is a (weak) Killing vector field on $(Z_0,\eta^Z)$ then $\phi_\ast(v)$ is also a (weak) Killing vector field on $(X_0,\eta^X)$. Therefore, the result follows by applying Proposition 7.4 and Theorem 7.3 from \cite{OW} after restricting the structure.
\end{proof}
The following result is clear.
\begin{lemma}\label{MoritaLemma2}
If $\eta_1$ and $\eta_2$ are equivalent Riemannian metrics on $X_1\rightrightarrows X_0$ then the crossed modules $(\mathfrak{o}_{m}^{\textnormal{w}}(X,\eta_1),\Gamma(A_X),\delta,D)$ and $(\mathfrak{o}_{m}^{\textnormal{w}}(X,\eta_2),\Gamma(A_X),\delta,D)$ agree.
\end{lemma}

Suppose that $X$ and $Y$ are Morita equivalent Lie groupoids so that there is a third Lie groupoid $Z$ with Morita fibrations $Z\to X$ and $Z\to Y$. From \cite[Prop. 6.3.1]{dHF2} we know that if $\eta^X$ is a Riemannian metric on $X$ then there exists a Riemannian metric $\eta^Z$ on $Z$ that makes the fibration $Z\to X$ Riemannian. We can slightly modify $\eta^Z$ by a cotangent averaging procedure so that we get another Riemannian metric $\tilde{\eta}^Z$ on $Z$ which descends to $Y$ defining a Riemannian metric $\eta^Y$ making of the fibration $Z\to Y$ Riemannian. It turns out that these pullback and pushforward constructions are well-defined and mutually inverse modulo equivalence of metrics. This is because $\eta^Z$ and $\tilde{\eta}^Z$ turn out to be equivalent, see the proof of Theorem 6.3.3 in \cite{dHF2}. In this case we refer to $(X,\eta^X)$ and $(Y,\eta^Y)$ as being \emph{Morita equivalent Riemannian groupoids}. It suggests a definition for Riemannian metrics over differentiable stacks. Namely, a stacky metric on the orbit stack $[X_0/X_1]$ presented by a Lie groupoid $X_1\rightrightarrows X_0$ is defined to be an equivalence class $[\eta]$ of a Riemannian metric $\eta$ on $X$. For further details the reader is recommended to visit \cite{dHF2}. 

Summing up, we get that:

\begin{theorem}\label{MoritaTheorem}
If $(X_1 \rightrightarrows X_0,\eta^X)$ and $(Y_1 \rightrightarrows Y_0,\eta^Y)$ are Morita equivalent Riemannian groupoids then the crossed modules $(\mathfrak{o}_{m}^{\textnormal{w}}(X),\Gamma(A_X),\delta,D)$ and $(\mathfrak{o}_{m}^{\textnormal{w}}(Y),\Gamma(A_Y),\delta,D)$ are isomorphic in the derived category of crossed modules. In consequence, the following quotient spaces are isomorphic as Lie algebras:
$$\mathfrak{o}_{m}^{\textnormal{w}}(X)/\textnormal{im}(\delta)\cong \mathfrak{o}_{m}^{\textnormal{w}}(Y)/\textnormal{im}(\delta).$$
\end{theorem}
\begin{proof}
This result is consequence of Lemmas \ref{MoritaLemma1} and \ref{MoritaLemma2} together with Theorem 7.4 and Corollaries 7.1 and 7.2 from \cite{OW}.
\end{proof}

Motivated by the previous result and \cite[Def. 8.1]{OW} we set up the following definition.
\begin{definition}
Let $(X_1 \rightrightarrows X_0,\eta)$ be a Riemannian groupoid. A \emph{geometric Killing vector field} on the quotient Riemannian stack $([X_0/X_1],[\eta])$ is defined to be an element of the quotient
$$\mathfrak{o}([X_0/X_1],[\eta]):=\mathfrak{o}_{m}^{\textnormal{w}}(X)/\textnormal{im}(\delta).$$
\end{definition}

On the one hand, if we consider proper étale Riemannian groupoids then geometric Killing vector fields recover the classical notions of Killing vector fields on both Riemannian manifolds and Riemannian orbifolds as defined for instance in \cite{BZ}. On the other hand, if we consider the Riemannian groupoid $\textnormal{Hol}(M,\mathcal{F})\rightrightarrows M$ associated to a regular Riemannian foliation $(M,\mathcal{F})$ then geometric Killing vector fields recover the notion of transverse Killing vector fields as defined in \cite[p. 84]{Mo}. Each of these particular cases has the property that the obtained algebra of geometric Killing vector fields is finite dimensional. We finish this section by proving that this is always the case if our quotient Riemannian stack is separated (i.e. it is presented by a proper Riemannian groupoid). Let us start by analyzing the Riemannian foliation groupoid case.  A \emph{foliation groupoid} is a Lie groupoid $X_1 \rightrightarrows X_0$ whose space of objects $X_0$ is Hausdorff and whose isotropy groups $X_x$ are discrete for all $x\in X_0$. For instance, every \'etale  Lie groupoid with Hausdorff objects manifold is a foliation groupoid. The converse is not true, however every foliation groupoid is Morita equivalent to an \'etale groupoid. As shown in \cite[Thm. 1]{CM} (see also \cite{C}), being a foliation groupoid is equivalent to the associated Lie algebroid anchor map $\rho:A_X\to TX_0$ being injective. As a consequence, the manifold $X_0$ comes with a regular foliation $\mathcal{F}$ tangent to the leaves of $\mathrm{im}(\rho)\subseteq TX_0$. Note that if $X_1 \rightrightarrows X_0$ is source-connected then the leaves of $\mathrm{im}(\rho)\subseteq TX_0$ coincide with the groupoid orbits. 

\begin{lemma}\label{Finite1}
	If $(X_1 \rightrightarrows X_0,\eta)$ is a Riemannian foliation groupoid with compact orbit space $X_0/X_1$ then the algebra of geometric Killing vector fields on $([X_0/X_1],[\eta])$ has finite dimension.
\end{lemma}
\begin{proof}
Let $\iota: T\hookrightarrow X_0$ be a complete transversal submanifold to the orbit foliation $\mathcal{F}$ of $X_1 \rightrightarrows X_0$  and consider its restricted groupoid $X_T\rightrightarrows T$. As $X_T\rightrightarrows T$ is étale and Morita equivalent to $X_1 \rightrightarrows X_0$ (see \cite[p. 136]{MM}), from Theorem \ref{MoritaTheorem} it follows that 
$$
		\mathfrak{o}([X_0/X_1],[\eta])=  \mathfrak{o}_m^{\textnormal{w}}(X_{T}, \iota^*\eta)/\textnormal{im}(\delta)\cong \mathfrak{o}(T)^{X}. 
$$
Here $\mathfrak{o}(T)^{X}$ denotes the transversal Killing vector fields that are invariant by the normal action. As $X_0/X_1$ is compact we have that $T$ has a finite number of connected components (see \cite[p. 135]{MM}), so that the dimension of $\mathfrak{o}(T)^{X}$ is finite by Theorem 3.3 from \cite[p. 238]{KN} (see also \cite[p. 85]{Mo}). That is, $\mathfrak{o}([X_0/X_1],[\eta])$ is finite dimensional.
\end{proof}

It is simple to see that similar arguments to those used in the proof of the previous Lemma work if we consider regular Riemannian groupoids instead of the Riemannian foliation ones. Here by regular we mean that the anchor map $\rho:A_X\to TX_0$ has locally constant rank. Therefore, as every proper groupoid is regular over a dense and open subset then one would expect that a similar result can be proven if we consider proper Riemannian groupoids. We show below that such a finite dimensional result is true by using the desingularization theorem for proper Riemannian groupoids proved in \cite[s. 6]{PTW}. Namely:

\begin{theorem}\label{FiniteThm2}
	Let $(X_1 \rightrightarrows X_0,\eta)$ be a proper Riemannian groupoid with compact orbit space $X_0/X_1$. Then the algebra of geometric Killing vector fields on $([X_0/X_1],[\eta])$ has finite dimension.
\end{theorem}
\begin{proof}
Let us denote by $(\tilde{X_1} \rightrightarrows \tilde{X_0},\tilde{\eta},\pi)$ the Riemannian desingularization of $(X_1 \rightrightarrows X_0,\eta)$, see Theorem 6.10 in \cite{PTW}. We have that $(\tilde{X_1} \rightrightarrows \tilde{X_0},\tilde{\eta})$ is a proper regular Riemannian groupoid and $\pi:\tilde{X_1}\to X_1$ is a proper Riemannian fibration which is almost-everywhere an isometry. As consequence of the functoriality properties described in \cite[s. 5.2]{PTW} it follows that any Lie groupoid automorphism of $X_1$ preserves its co-dimensional stratum data (see \cite[s. 3]{PTW}), so that they can be lifted to Lie groupoid automorphism of $\tilde{X_1}$. That is, we can lift multiplicative vector fields on $X_1$ to multiplicative vector fields on $\tilde{X_1}$ by lifting their 1-parametric (local) families of Lie groupoid automorphisms determined by their flows. In particular, by using both Proposition 6.13 and Theorem 6.14 from \cite{PTW} we get that weak multiplicative Killing vector fields on $(X_1,\eta)$ can be lifted to weak multiplicative Killing vector fields on its desingularization $(\tilde{X_1},\tilde{\eta})$, thus obtaining an surjective algebra homomorphism 
$$[\pi]: \mathfrak{o}([\tilde{X_0}/\tilde{X_1}],[\tilde{\eta}])\to \mathfrak{o}([X_0/X_1],[\eta]).$$

By arguing as in Proposition 6.3.2 from \cite{dHF2} and by using the classification of regular Lie groupoids given in \cite{Moe}, since the groupoid $\tilde{X_1}$ is regular we may assume that it fits into an extension of Riemannian groupoids $(K,\iota^\ast\tilde{\eta}) \xrightarrow[]{\it \iota} (\tilde{X_1},\tilde{\eta})\xrightarrow[]{\it q} (E,q_\ast\tilde{\eta})$ over $\tilde{X_0}$ where $K$ is a bundle of connected Lie groups (i.e. a Lie groupoid whose source and target maps agree), $E$ is a foliation groupoid, $\iota$ is a groupoid Riemannian embedding and $\pi$ is a groupoid Riemannian submersion with connected fibers ($\tilde{\eta}$ possibly needs to be averaged in order to define an equivalent groupoid metric which descends to the quotient). This in turn induces an extension of algebras 
$$0\to \mathfrak{o}(\tilde{X_0},\iota^\ast\tilde{\eta})\to \mathfrak{o}([\tilde{X_0}/\tilde{X_1}],[\tilde{\eta}])\to \mathfrak{o}([\tilde{X_0}/E],[q_\ast\tilde{\eta}])\to 0.$$

As $X_0/X_1$ is compact and $\pi$ is proper and surjective it follows $\tilde{X_0}/\tilde{X_1}$ is also compact and from Lemma \ref{Finite1} together with Theorem 3.3 from \cite[p. 238]{KN} (see also \cite[p. 85]{Mo}) it follows that both $\mathfrak{o}([\tilde{X_0}/E],[q_\ast\tilde{\eta}])$ and $\mathfrak{o}(\tilde{X_0},\iota^\ast\tilde{\eta})$ have finite dimension so that $\mathfrak{o}([\tilde{X_0}/\tilde{X_1}],[\tilde{\eta}])$ has also finite dimension. That is, $\mathfrak{o}([X_0/X_1],[\eta])$ is finite dimensional since $[\pi]$ is surjective.
\end{proof}

We end the paper by describing some features about our notion of geometric Killing vector field which are motivated by some results proved in \cite{CMS}. Let $(X_1\rightrightarrows X_0,\eta)$ be a Riemannian groupoid. A vector field $\xi_1$ on $X_1$ is said to be \emph{projectable of Killing type} if there exists a vector field $v_1$ on $(X_0,\eta^{(0)})$ such that the following conditions are satisfied:
\begin{enumerate}
	\item[a.] $\xi_1$ and $v_1$ are both $s_X$-related and $t_X$-related, and
	\item[b.] if $\varPhi_\tau$ and $\varphi_\tau$ respectively denote the (local) flows of $\xi_1$ and $v_1$ then the induced map $\overline{d\varphi_\tau}:\nu(\mathcal{O}_x)\to \nu(\mathcal{O}_{\varphi_\tau(x)})$ is a linear isometry for each groupoid orbit $\mathcal{O}_x$ for which $\varphi_\tau(x)$ is defined.
\end{enumerate}

Compare with \cite[Def. 4.6]{CMS}. We shall also refer to $v_1$ as a \emph{weak Killing vector field} on $(X_0,\eta^{(0)})$. The space of vector fields of Killing type on $X_1$ will be denoted by $\Gamma^\textnormal{proj}_\eta(X_1)$. Firstly, note that Condition b. makes sense because the local flows $\varPhi_\tau$ and $\varphi_\tau$ commute with both $s_X$ and $t_X$ as consequence of Condition a. Secondly, it follows that $\xi_1$ is also a weak Killing vector field on $(X_1,\eta^{(1)})$ in the sense that $\overline{d\varPhi_\tau}:\nu(G_{\mathcal{O}_x})\to \nu(G_{\mathcal{O}_{\varphi_\tau(x)}})$ is a linear isometry since $s_X$ and $t_X$ are Riemannian submersions. Actually, it is simple to see that a similar property holds true for the vector field $\xi_2(p,q)=(\xi_1(p),\xi_1(q))$ on $(X_2,\eta^{(2)})$ induced by $\xi_1$. 

We explain below how proper Haar measure systems on proper Riemannian groupoids determine projections from the space of projectable vector fields of Killing type on $(X_1,\eta^{(1)})$ to the space of weak multiplicative Killing vector fields on $(X_1\rightrightarrows X_0,\eta)$. This can be done by following \cite[s. 2.5]{CS}.

\begin{theorem}\label{ExistenceWeak}
Let $(X_1\rightrightarrows X_0,\eta)$ be a proper Riemannian groupoid. Then any proper Haar measure system $\lbrace \mu^x\rbrace$ for $X_1\rightrightarrows X_0$ induces a linear map from $\Gamma^\textnormal{proj}_\eta(X_1)$  to $\mathfrak{o}_{m}^{\textnormal{w}}(X,\eta)$.
\end{theorem}
\begin{proof}
	Let us pick a projectable vector field of Killing type $\xi_1$ on $(X_1,\eta^{(1)})$. From \cite[s. 2.5]{CS} we know that we can construct a multiplicative vector field $\xi:X_1\to TX_1$ by taking the average with respect to $\lbrace \mu^x\rbrace$:
	$$\xi(p)=\int_{a\in t^{-1}(s(p))}dm_{(pa,a^{-1})}(\xi_1(ap),di_a(\xi_1(a))\mu(a).$$
	This vector field is $s_X$-related with the vector field $v$ on $X_0$ defined as
	$$v(x)=\int_{a\in t^{-1}(x)}dt_a(\xi_1(a))\mu(a).$$
	Therefore, our result will follow once we prove that the flow of $v$ is a local transversal isometry of $(X_0,\eta)$ since $\xi$ and $v$ are $s_X$-related. However, as $\xi_1$ is a weak Killing vector field on $(X_1,\eta^{(1)})$, this follows from a straightforward computation after noting that the flow of $v$ is given by $\varphi^v_r(x)=\int_{a\in t^{-1}(x)}(t\circ \varphi^{\xi_1}_r)(a)\mu(a)$ and $\overline{dt}:\nu(G_{\mathcal{O}})\to \nu(\mathcal{O})$ is a fiberwise isometry since $t_X$ is a Riemannian submersion. Hence, the assignment $\xi_1\mapsto (\xi,v)$ establishes the desired projection. 
\end{proof}

As consequence of Remark \ref{WeakKillingNerve} we may conclude that the previous result is actually a fact that can be proven for any $n$-metric $\eta^{(n)}$ on $X_n$. 

In \cite[s. 4.4]{CMS} it was studied the ``normal'' bundle $\vartheta:=TX_0/\rho(A_X)$ where $\rho:A_X\to TX_0$ is the anchor map of the Lie algebroid $A_X$ of $X_1\rightrightarrows X_0$. This is a smooth vector bundle only in the regular case. Its space of sections is defined to be the quotient $\Gamma(\vartheta):=\mathfrak{X}(X_0)/\textnormal{im}(\rho)$. A section $[v]\in \Gamma(\vartheta)$ is called \emph{invariant} if there exists a vector field $\xi$ on $X_1$ which is both $s_X$-related and $t_X$-related to $v$. The resulting space of invariant elements is denoted by $\Gamma(\vartheta)^{\textnormal{inv}}$. Recall that for each section $\alpha\in \Gamma(A_X)$ we have an associated multiplicative vector field $\delta(\alpha)=(\alpha^r-\alpha^l,\rho(\alpha))$ on $X_1\rightrightarrows X_0$. From Lemma 4.7 in \cite{CMS} it follows that there is a natural linear map from $\mathfrak{X}_{m}(X)/\textnormal{im}(\delta)$ to $\Gamma(\vartheta)^{\textnormal{inv}}$ which associates to a multiplicative vector field $\xi$ on $X_1$ the class modulo $\textnormal{im}(\rho)$ of the vector field $v$ on $X_0$ associated with $\xi$. Furthermore, if $X_1\rightrightarrows X_0$ is proper then the latter map induces an isomorphism $\mathfrak{X}_{m}(X)/\textnormal{im}(\delta)\cong \Gamma(\vartheta)^{\textnormal{inv}}$, see Theorem 6.1 in \cite{CMS}. 

Let $(X_1\rightrightarrows X_0,\eta)$ be a proper Riemannian groupoid. Motivated by the previous facts we define the space of \emph{Killing invariant sections} $\Gamma_\eta(\vartheta)^{\textnormal{inv}}$ as the set of sections $[v]\in \Gamma(\vartheta)$ for which there exists a projectable vector field of Killing type $\xi$ on $X_1$ over $v$. That is, Conditions a. and b. in the definition of projectable vector field of Killing type are satisfied with $\xi$ and $v$. Therefore, as consequence of Lemma 4.7 and Theorem in \cite{CMS} we immediately get that:
\begin{proposition}
There exists a natural isomorphism between $\mathfrak{o}_{m}^{\textnormal{w}}(X,\eta)/\textnormal{im}(\delta)$ and $\Gamma_\eta(\vartheta)^{\textnormal{inv}}$.
\end{proposition}

In particular, Theorem \ref{ExistenceWeak} provides us with a method to construct geometric Killing vector fields over the quotient Riemannian stack $([X_0/X_1],[\eta])$.\\

\noindent {\bf Conflict of interests.} The authors declare that they have no conflict of interests.


\begin{thebibliography}{WWW}
	
\bibitem{AB} M. M. Alexandrino, R. G. Bettiol: \emph{Lie groups and geometric aspects of isometric actions}, Springer, Cham, (2015).

\bibitem{An} C. Angulo: \emph{A cohomology theory for Lie 2-algebras and Lie 2-groups}, PhD thesis, Instituto de Matemática e Estatística, University of São Paulo, São Paulo, (2018).

\bibitem{AB} D. M. Austin, P. J. Braam: \emph{Morse-{B}ott theory and equivariant cohomology}, The {F}loer memorial volume, Progr. Math., {\bf 133} Birkh\"{a}user, Basel (1995) , 123--183.
%\bibitem{A} C. Angulo: \emph{A new cohomology theory for strict Lie 2-algebras}, Commun. Contemp. Math., (2021) 2150017, 1--42. 

\bibitem{Ba} J. Baez: \emph{Higher Yang–Mills Theory}, preprint:  arXiv:hep-th/0206130v2, (2002).


%\bibitem{behrend2004cohomology}
%Kai Behrend.
%\newblock Cohomology of stacks.
%\newblock {\em Intersection theory and moduli, ICTP Lect. Notes}, 19:249--294,
%2004.

%\bibitem{berwick2020lie}
%Daniel Berwick-Evans and Eugene Lerman.
%\newblock Lie 2-algebras of vector fields.
%\newblock {\em Pacific Journal of Mathematics}, 309(1):1--34, 2020.

\bibitem{BD} J. C. Baez, A. D. Lauda: \emph{Higher-dimensional algebra. V: 2-Groups}, Theory Appl. Categ., {\bf 12} (2004), 423--491.

%\bibitem{Ablaom-CArtanAlgebroids}
%Anthony Blaom.
%\newblock Geometric structures as deformed infinitesimal symmetries.
%\newblock {\em Transactions of the American Mathematical Society},
%358(8):3651--3671, 2006.

%\bibitem{bos2013lecture}
%Rogier Bos.
%\newblock Lecture notes on groupoid cohomology, 2013.

\bibitem{BZ} A. V. Bagaev, N. I. Zhukova: \emph{The isometry groups of {R}iemannian orbifolds}, Sibirsk. Mat. Zh., {\bf 48} (2007) no. 4, 723--741.

\bibitem{OBT} A. Barbosa-Torres, C. Ortiz, \emph{Equivariant cohomology of stacky Lie group actions}, Private communication.

\bibitem{B} G. E. Bredon: \emph{Introduction to compact transformation groups}, Pure and Applied Mathematics, Vol. 46., Academic Press, New York-London, (1972).

\bibitem{BS} R. Brown, C. Spencer: \emph{$G$-groupoids, crossed modules and the fundamental groupoid of a topological group}, Indag. Math., {\bf 38} (1976) no. 4, 296--302.


\bibitem{CCK} S. Chatterjee, A. Chaudhuri, P. Koushik: \emph{Atiyah sequence and gauge transformations of a principal 2-bundle over a Lie groupoid}, J. Geom. Phys., {\bf 176}  (2022) No. 104509.

\bibitem{C} M. Crainic, \emph{Differentiable and algebroid cohomology, Van Est isomorphisms, and characteristic classes,} Commentari Math. Helvetici 78 (2003), 681--721.

\bibitem{CM} M. Crainic, I. Moerdijk, \emph{Foliation groupoids and their cyclic homology,} Adv. Math., {\bf 157} (2001) no. 2, 177--197.

\bibitem{CMS} M. Crainic, J. N. Mestre, I. Struchiner, \emph{Deformations of {L}ie groupoids}, Int. Math. Res. Not. IMRN,  (2021) (21), 7662--7746.


\bibitem{CS} M. Crainic, I. Struchiner: \emph{On the linearization theorem for proper {L}ie groupoids}, Ann. Sci. \'{E}c. Norm. Sup\'{e}r. (4), {\bf 46}  (2013) no. 5, 723--746.

\bibitem{dH} M. del Hoyo: \emph{Lie groupoids and their orbispaces}, Port. Math., {\bf 70}  (2012) no. 2, 161--209.

\bibitem{dHdM} M. del Hoyo, M. de Melo: \emph{Geodesics on Riemannian stacks}, Transform. Groups, {\bf 27} (2022) no. 2, 403--427.

\bibitem{dHdM2} M. del Hoyo, M. de Melo: \emph{On invariant linearization of {L}ie groupoids}, Lett. Math. Phys., {\bf 111} (2021) no. 4, Paper No. 112, 14.

\bibitem{dHF} M. del Hoyo, R. Fernandes: \emph{Riemannian metrics on {L}ie groupoids}, J. Reine Angew. Math., {\bf 735}  (2018), 143--173. 

\bibitem{dHF2} M. del Hoyo, R. Fernandes: \emph{Riemannian metrics on differentiable stacks}, Math. Z., {\bf 292} (2019) no. 1-2, 103--132.

\bibitem{FMM} J. Faria Martins, A. Mikovi\'{c}:  {\it Lie crossed modules and gauge-invariant actions for 2-{BF} theories}, Adv. Theor. Math. Phys.,  {\bf 15} (2011) n. 4, 1059--1084.

\bibitem{GGHR} E. Gallego, L. Gualandri, G. Hector, A. Revent\'{o}s:  {\it Groupo\"{\i}des riemanniens}, Publ. Mat.,  {\bf 33} (1989) n. 3, 417--422.

\bibitem{GZ} A. Garmendia, M. Zambon:  {\it Quotients of singular foliations and {L}ie 2-group actions}, J. Noncommut. Geom.,  {\bf 15} (2021) n. 4, 1251--1283.

%\bibitem{HOV} S. Herrera, C. Ortiz, F. Valencia: \emph{Principal $2$-bundles over Lie groupoids}, work in progress.


\bibitem{hsz} B. Hoffman, R. Sjamaar: \emph{Stacky {H}amiltonian actions and symplectic reduction}, Int. Math. Res. Not. IMRN, {\bf 20} (2021) no. 20, 15209--15300.

\bibitem{IK} S. Illman, M. Kankaanrinta: \emph{Three basic results for real analytic proper {$G$}-manifolds}, Math. Ann., {\bf 316} (2000) no. 1, 169--183.

%\bibitem{kobayashi}
%S.~{Kobayashi} and K.~{Nomizu}.
%\newblock {Foundations of differential geometry. Vol. I.}
%\newblock {New York-London: Interscience Publishers, a division of John Wiley
%	\& Sons. XI, 329 p. (1963).}, 1963.

%\bibitem{kobayashitransGroups}
%S.~Kobayashi.
%\newblock {\em Transformation Groups in Differential Geometry}.
%\newblock Classics in Mathematics. Springer Berlin Heidelberg, 2012.

%\bibitem{kotov2018integration}
%Alexei Kotov and Thomas Strobl.
%\newblock Integration of quadratic lie algebroids to riemannian cartan--lie
%groupoids.
%\newblock {\em Letters in Mathematical Physics}, 108(3):737--756, 2018.

%\bibitem{kotovStrobl}
%Alexei Kotov and Thomas Strobl.
%\newblock Lie algebroids, gauge theories, and compatible geometrical
%structures.
%\newblock {\em Reviews in Mathematical Physics}, 31(04):1950015, 2019.



\bibitem{K} M. Kankaanrinta: \emph{Equivariant collaring, tubular neighbourhood and gluing theorems for proper {L}ie group actions}, Algebr. Geom. Topol., {\bf 7} (2007), 1--27.

\bibitem{KN} S. Kobayashi, K. and Nomizu: \emph{Foundations of differential geometry. {V}ol. {I}}, John Wiley \& Sons, Inc., New York,  (1996).

%\bibitem{LSZ} Z. Liu, Y. Sheng, T. Zhang: \emph{Deformations of {L}ie 2-algebras}, J. Geom. Phys., {\bf 86} (2014), 66--80.



\bibitem{MX} K. Mackenzie, P. Xu: \emph{Classical lifting processes and multiplicative vector fields}, Quart. J. Math. Oxford Ser. (2), {\bf 49}  (1998), 59--85.

\bibitem{Ma} K. Mackenzie: \emph{The general theory of {Lie} groupoids and {Lie} algebroids}, {\bf 213} Lond. Math. Soc. Lect. Note Ser., Cambridge: Cambridge University Press, (2005).



\bibitem{Me} A. Medina: \emph{Groupes de {L}ie munis de m\'{e}triques bi-invariantes}, Tohoku Math. J. (2), {\bf 37} (1985) no. 4, 405--421.

\bibitem{Mein} E. Meinrenken, \emph{Euler-like vector fields, normal forms, and isotropic embeddings}, Indag. Math. (N.S.) 32 (1) (2021), 224--245.

\bibitem{Mi} J. Milnor: \emph{Curvatures of left invariant metrics on {L}ie groups}, Advances in Math., {\bf 21} (1976) no. 3, 293--329.

\bibitem{Moe} I. Moerdijk: \emph{Lie groupoids, gerbes, and non-abelian cohomology}, $K$-Theory, {\bf 28} (2003) no. 3, 207--258.


\bibitem{MM} I. Moerdijk, J. Mr\v{c}un: \emph{Introduction to foliations and {L}ie groupoids}, Cambridge Studies in Advanced Mathematics, Cambridge University Press, Cambridge 91, (2003).

\bibitem{Mo} P. Molino: \emph{Riemannian foliations}, Progress in Mathematics: Translated from the French by Grant Cairns,
With appendices by Cairns, Y. Carri\`ere, \'{E}. Ghys, E. Salem and V. Sergiescu, Birkh\"{a}user Boston, Inc., Boston, MA, (1998).

\bibitem{O'n} B. O'Neill: \emph{The fundamental equations of a submersion}, Michigan Math. J., {\bf 13}  (1966), 459--469.

\bibitem{OW} C. Ortiz, J. Waldron: \emph{On the {L}ie 2-algebra of sections of an {$\mathcal{LA}$}-groupoid}, J. Geom. Phys., {\bf 145} (2019), No. 103474, 34.

\bibitem{OV} C. Ortiz, F. Valencia: \emph{Morse theory on Lie groupoids}, preprint:  arXiv:math.DG/2207.07594, (2022).

\bibitem{PPT} M. J. Pflaum, H. Posthuma, X. Tang: \emph{Geometry of orbit spaces of proper {L}ie groupoids}, J. Reine Angew. Math., {\bf 694} (2014), 49--84.

\bibitem{PTW} H. Posthuma, X. Tang, K. Wang: \emph{Resolutions of proper {R}iemannian {L}ie groupoids}, Int. Math. Res. Not. IMRN, (2021) no. 2, 1249--1287.


\bibitem{SW} A. Schmeding, C. Wockel: \emph{The {L}ie group of bisections of a {L}ie groupoid}, Ann. Global Anal. Geom., {\bf 48} (2015) no. 1, 87--123.

\bibitem{W} A. G. Wasserman: \emph{Equivariant differential topology}, Topology, {\bf 8} (1967), 127--150.



\end{thebibliography}
\end{document}